\documentclass[reqno,english]{amsart}
\usepackage{amsfonts,amsmath,latexsym,verbatim,amscd,mathrsfs,color,array}
\usepackage{amsmath,amssymb,amsthm,amsfonts}
\usepackage{graphicx,color}

\usepackage[colorlinks=true,citecolor=blue]{hyperref}

\usepackage[yyyymmdd,hhmmss]{datetime}

\newcommand{\ass}{\quad\mbox{as}\quad}
\usepackage{graphicx}

\newcommand{\EE}{{\mathcal E}  }
\newcommand{\inn}{{\quad\hbox{in } }}
\newcommand{\onn}{{\quad\hbox{on } }}
\newcommand{\ttt}{\tilde }
\newcommand{\TT}{{\mathcal T}  }

\newcommand{\nn}{ {\nabla}  }

\newcommand{\pp}{ {\partial} }

\newcommand{\vp}{\varphi}

\newcommand{\RR}{{{\mathcal R}}}

\newcommand{\C}{{\mathbb C}}
\newcommand{\R} {\mathbb R}
\newcommand{\Z} {\mathbb Z}
\newcommand{\cuad}{{\sqcap\kern-.68em\sqcup}}

\newcommand{\DD}{{\mathcal D}}
\newcommand{\Rem}{{\mathcal R}_0}

\newcommand{\KK}{{\mathcal K}}

\newcommand{\foral}{\quad\mbox{for all}\quad}

\newcommand{\be}{\begin{equation}}
\newcommand{\ee}{\end{equation}}

\newcommand{\la}{\lambda}

\newcommand{\equ}[1]{(\ref{#1})}

\renewcommand{\Re}{\mathop{\rm Re}}
\renewcommand{\Im}{\mathop{\rm Im}}

\renewcommand{\div}{\mathop{\rm div}}
\newcommand{\curl}{\mathop{\rm curl}}

\newtheorem{lemma}{Lemma}[section]
\newtheorem{prop}{Proposition}[section]
\newtheorem{theorem}{Theorem}

\newtheorem{remark}{Remark}[section]
\newtheorem{ch}{Check}
\newcommand{\bremark}{\begin{remark} \em}
\newcommand{\eremark}{\end{remark} }

\long\def\hide#1{}

\long\def\noanot#1{}

\definecolor{redd}{RGB}{200,0,0}
\definecolor{mo1}{rgb}{0.8,0,0.8}
\definecolor{green1}{rgb}{0.1,0.7,0.1}
\definecolor{g2}{rgb}{0,0.5,0}

\long\def\elim#1{{\color{red} ELIMINAR\\ #1}}

\long\def\elim#1{}

\numberwithin{equation}{section}

\title[Blow-up in the axially symmetric harmonic map flow]{Blow-up for the 3-dimensional axially symmetric   harmonic map flow into $S^2$}

\author[J. Davila]{Juan Davila}
\address{\noindent
Instituto de Matem\'aticas, Universidad de Antioquia, Calle 67, No. 53--108, Medell\'\i n, Colombia,
and  Departamento de
Ingenier\'{\i}a  Matem\'atica-CMM   Universidad de Chile,
Santiago 837-0456, Chile}
\email{jdavila@dim.uchile.cl}

\author[M. del Pino]{Manuel del Pino}
\address{\noindent   Department of Mathematical Sciences University of Bath,
Bath BA2 7AY, United Kingdom \\
and  Departamento de
Ingenier\'{\i}a  Matem\'atica-CMM   Universidad de Chile,
Santiago 837-0456, Chile}
\email{m.delpino@bath.ac.uk}

\author[C. Pesce]{Catalina Pesce}
\address{\noindent  Departamento de
Ingenier\'{\i}a  Matem\'atica-CMM   Universidad de Chile,
Santiago 837-0456, Chile}
\email{catalina.pesce.r@ing.uchile.cl}

\author[J. Wei]{Juncheng Wei}
\address{\noindent
Department of Mathematics,
University of British Columbia, Vancouver, B.C., Canada, V6T 1Z2}
\email{jcwei@math.ubc.ca}

\begin{document}

\begin{abstract}
We construct finite time blow-up solutions to the 3-dimensional harmonic map flow into the sphere $S^2$,
\begin{align*} u_t & = \Delta u + |\nabla u|^2 u \quad \text{in } \Omega\times(0,T)
\\
u &= u_b \quad \text{on } \partial \Omega\times(0,T)
\\
u(\cdot,0) &= u_0 \quad  \text{in } \Omega  ,
\end{align*}
with $u(x,t): \bar \Omega\times [0,T) \to S^2$.
Here $\Omega$ is a bounded, smooth axially symmetric domain in $\mathbb{R}^3$.
We prove that for any circle $\Gamma \subset \Omega$ with the same axial symmetry,  and any sufficiently small $T>0$ there exist  initial and boundary conditions such that $u(x,t)$ blows-up exactly at time $T$ and precisely on the curve $\Gamma$, in fact
$$
|\nn u(\cdot ,t)|^2 \rightharpoonup  |\nn u_*|^2   +  8\pi \delta_\Gamma \ass t\to T .
$$
for a regular function $u_*(x)$, where $\delta_\Gamma$  denotes the Dirac measure supported on the curve. This the first example of a blow-up solution with a space-codimension 2 singular set, the maximal dimension predicted in the partial regularity theory by Chen-Struwe and Cheng \cite{chen-struwe,cheng}.

\end{abstract}

\maketitle


\section{Introduction and main result}

Let $\Omega$ be a bounded domain  in $\R^3$ with smooth boundary $\pp\Omega$.  We denote by $S^2$ the standard 2-sphere.  We consider the {\em harmonic map flow}
for maps from $\Omega$ into $S^2$, given by the semilinear parabolic equation
\be
\label{har flow0} \left \{ \begin{aligned}
u_t = \Delta u + |\nabla u|^2 u \quad &\text{in } \Omega\times(0,T)\\
u = u_b  \quad &\text{on } \pp\Omega\times(0,T)\\
u(\cdot,0) = u_0  \quad & \text{in } \Omega
\end{aligned}\right.\ee
for a function $u:\Omega \times [0,T)\to S^2$. Here $u_0:\bar \Omega \to S^2$ is a given smooth map and  $\vp= u_0\big|_{\pp\Omega}$.  Local existence and uniqueness of a classical solution follows from the pioneering work by Eells and Sampson \cite{Ells} and K.C. Chang \cite{chang}. Equation \equ{har flow0} formally corresponds to the negative $L^2$-gradient flow for the Dirichlet energy $\int_\Omega |\nabla u|^2 dx$. This energy is decreasing
along smooth solutions $u(x,t)$:
$$
\frac {\pp}{\pp t} \int_\Omega |\nabla u(\cdot, t)|^2  = -\int_\Omega |u_t(\cdot, t)|^2.
$$
Chen-Struwe \cite{chen-struwe} found a global $H^1$-weak solution in any dimension. In the two-dimensional case \\ $\Omega\subset \R^2\mapsto S^2$ this solution can only become singular at a finite number of points in space-time \cite{Struwe}.

\medskip
If $T>0$ designates the first instant at which smoothness of \equ{har flow0} is lost, standard parabolic regularity leads to the fact that
$$
 \|\nabla u(\cdot, t)\|_\infty \, \to \,  +\infty \ass t\uparrow T.
$$
In the two-dimensional case, substantial knowledge on the possible blow-up structure  has been obtained in \cite{Ding-Tian,Lin-Wang1,Qing1,Qing-Tian,Struwe,Topping2}. Blow-up takes place only about a finite number of points $q_1,\ldots, q_k$, around which
the approximate form  $u(x,t) \approx  U\left (\frac{x-\xi(t)}{\la(t)}  \right )$ with $\la(t)\to 0$ where $U$ is a finite-energy harmonic map, namely a solution of $$\Delta U +|\nn U|^2 U=0, \quad |U|\equiv 1 \inn \R^2 , \quad , \quad \int_{\R^2} |\nabla U|^2 <+\infty $$
 and $\la(t)\to 0$ as $t\to T$. Moreover (up to subsequences), we have
\be \label{aa2}
|\nabla u(\cdot, t)|^2
\ \rightharpoonup  \ |\nabla u_*|^2 +   \sum_{i=1}^k 4\pi m_i\,\delta_{q_i} \ass t\to T,
\ee
 for some positive integers $m_i$ where $\delta_q$ denotes the unit Dirac mass at $q$.

\medskip
Less is known in the higher dimensional case $\Omega\subset \R^n\mapsto S^2$ in problem \equ{har flow0}. Chen-Struwe and Cheng \cite{chen-struwe,cheng} have proven
that the blow-up set in $\Omega$ is at most $(n-2)$-dimensional in the Hausdorff sense.
More refined information on the singular set
has been derived by Lin and Wang in \cite{Lin-Wang5}, see also \cite{LinLibro}.

\medskip
While various important blow-up classification results  are available, finding solutions explicitly exhibiting blow-up behavior has been
rather difficult. In fact, in the two-dimensional case they were even believed not to exist, see \cite{Chang-Ding-Ye}.
The first example of a blowing-up solution in the case $\Omega =B_2 \subset \R^2 $, the unit two-dimensional ball was found
by Chang-Ding-Ye   \cite{Chang-Ding-Ye} in the
{\em 1-corrotational symmetry class},
$$
u(x,t)  =  \left (\begin{matrix}    e^{ i\theta} \sin v(r,t)   \\  \cos v(r,t) \end{matrix}    \right ) , \quad x= re^{i\theta}  .
$$
where $v(r,t)$ is a scalar function.
System \equ{har flow0} reduces to the radial scalar equation
\begin{align*}
v_t = v_{rr} + \frac{v_r}r  -  \frac {\sin v\cos v} {r^2}, \quad v(0,t)=0, \quad  r\in (0,1).
\end{align*}
Suitable initial and boundary conditions and the use of barriers lead to finite-time blow-up at some $T>0$
in the form
$v(r,t) \approx w(\frac{r}{\lambda(t)}) $ with
$$
w(\rho) = \pi - 2\arctan (\rho).
$$
Van den Berg, Hulshof and  King  \cite{bhk} formally found that generically,
\begin{align*}
\la(t) \approx  \kappa \frac{ T-t}{|\log (T-t)|^2} \ass t\to T.
\end{align*}
for some $\kappa >0$. Raphael and Schweyer \cite{rs1}  rigorously constructed an entire 1-corrotational solution with this blow-up rate.
At the level of $u$, the solutions mentioned above have the form
$$
u(x,t)  \approx   W\left(\frac x{\la(t)}   \right )
$$
where $W(y)$ is the canonical 1-corrotational harmonic map
\begin{equation}
\label{U00}
W(y)  =    \frac 1{1+ |y|^2} \left (\begin{matrix}    2y   \\  |y|^2 -1 \end{matrix}    \right ) , \quad   y\in \R^2 .
\end{equation}
which satisfies
$$
\int_{\R^2} |\nabla W |^2 = 4\pi, \quad  W (\infty) =  {\bf e}_3 ,
$$
where
\begin{align}
\label{e123}
{\bf e}_1 =\left (\begin{matrix}   1 \\ 0\\   0 \end{matrix}    \right ), \quad {\bf e}_2 =\left (\begin{matrix}   0 \\ 1\\   0 \end{matrix}    \right ),\quad {\bf e}_3 =\left (\begin{matrix}   0 \\ 0\\   1 \end{matrix}    \right ).
\end{align}
We achieved in \cite{ddw} the first construction of a blow-up solution without symmetries in \equ{har flow0} in the case
$\Omega \subset \R^2 \mapsto S^2$:  for an arbitrary $\Omega\subset \R^2$,
given points $q_1,\ldots, q_k\in\Omega$ and $u_b ={\bf e}_3$ there is for any sufficiently small $T>0$ a solution $u(x,t)$  with precisely these $k$ blow-up points which, consistently with \equ{aa2}, satisfies
\be \label{aa3}
|\nabla u(\cdot, t)|^2
\ \rightharpoonup  \ |\nabla u_*|^2 +   \sum_{i=1}^k 4\pi \,\delta_{q_i} \ass t\to T,
\ee
which near each $q_j$ and after a rigid constant rotation has the approximate form
$$
u(x,t) \approx  W\left(\frac {x-q_j}{\la_j(t)}   \right ), \quad \la_j(t) =   \kappa_j \frac{ T-t}{|\log (T-t)|^2} \ass t\to T.
$$
Part of the difficulty in the construction is due to the {\em instability}  of the blow-up phenomenon here described once the 1-corrotational symmetry is violated, see \cite{ddw}. This instability had been numerically conjectured in \cite{williams}.

\medskip
In the case $\Omega \subset \R^3\to S^2$ only one example has been know, again in the 1-corrotational class and $\Omega =B_3$, the unit ball in $\R^3$. In this case the ansatz takes the form
$$
u(x,t)  =  \left (\begin{matrix}    e^{ i\theta} \sin v(r,z,t)   \\  \cos v(r,z,t) \end{matrix}    \right ) , \quad x=
\left (\begin{matrix}    re^{ i\theta}    \\  z \end{matrix}    \right )
$$
System \equ{har flow0} reduces to the scalar equation
\be\label{111}
v_t = v_{rr} +\frac{v_r}r + v_{zz}   -  \frac {\sin v\cos v} {r^2}, \quad v(0,z,t)=0, \quad r\in (0,1).
\ee
Adapting the barrier method in \cite{Chang-Ding-Ye}, Grotowski \cite{grotowski}  found boundary and initial conditions and a  solution to
\equ{111} that blows up on a subset of the $z$-axis $r=0$.
(See a related result in \cite{grotowski2}.)
No information on the structure (or dimension) of this set or on the blow-up rate is provided.

\medskip
In this paper we construct the first example of a solution with a 1-dimensional blow-up set in an arbitrary axisymmetric bounded domain
$\Omega \subset \R^3$. We observe that  this example saturates the estimate for the dimension $n-2$ of the singular set found in \cite{cheng} (for $n=3$).

\medskip
Before stating our main result we introduce the setting we will consider.
We say that $\Omega\subset \R^3$ is an axisymmetric domain if it can be expressed in the form
\begin{align}
\label{axi1}
\Omega =  \{ ( re^{i\vartheta} , z) \ /\ (r,z) \in \DD, \quad \vartheta\in [0,2\pi] \} .
\end{align}
where $\DD\subset \{(r,z) /\ r\ge 0\}\subset \R^2 $. When $\Omega$ is axisymmetric, it is natural to look for solutions  of \equ{har flow0}  with the same axial symmetry, namely
\begin{align*}
u(x,t) =  \ttt u ( r,z,t) , \quad x = ( re^{i\vartheta} , z), \quad (r,z)\in \DD ,
\end{align*}
for a function $\ttt u : \DD \times (0,T) \to S^2$.

\medskip
We fix in what follows and axisymmetric, smooth and bounded domain $\Omega$  of the form \equ{axi1}
Let us consider  a point $(r_0,z_0) \in \DD$ with $r_0>0$  and let $\Gamma$ be the curve inside $\Omega$ given by the copy of $S^1$,
\be\label{Gamma} \Gamma := \{ (r_0e^{i\vartheta}, z_0) \ / \ \vartheta \in [0,2\pi) \} \subset \Omega \ee

\begin{theorem}\label{teo1}
Let $\Omega \subset \R^3$ be an axisymmetric domain and consider problem $\equ{har flow0}$ with boundary condition $u_b \equiv {\bf e}_3$. Then for all
sufficiently small $T>0$ there exists an initial condition and a solution $u(x,t)$ that blows-up exactly on the curve $\Gamma$ in $\equ{Gamma}$,
with a profile of the form
$$
u(x,t) = W \left (\frac { (r,z)- \xi(t) } {\la(t)}     \right ) + u_*(x) , \quad x= (re^{i\vartheta},  z) \ass t\to T.
$$
where $W(y)$ is the standard two-dimensional 1-corrotational map  $\equ{U00}$,  $u_* \in H^1(\Omega)$,
 $\la(t)\to 0$ and  $\xi(t) \to (r_0,z_0)$.

\end{theorem}
The proof provides much finer information on the asymptotic profile. In particular
we have, analogously to \equ{aa3},
\begin{align*}
|\nabla u(\cdot, t)|^2
\ \rightharpoonup  \ |\nabla u_*|^2 +   4\pi \,\delta_{\Gamma} \ass t\to T,
\end{align*}
with $\delta_\Gamma$ the uniform Dirac measure on the curve $\Gamma$. Moreover, writing $\xi(t) = (\xi_1(t), \xi_2(t))$ we have the
asymptotic expressions
\begin{align*}
\left\{
\begin{aligned}
\xi_1(t) &=   \sqrt{ r_0^2 +  2(T-t) } +  O ( (T-t)^{1+\sigma}) , \\
\xi_2(t) &  =  z_0+   O ( (T-t)^{1+\sigma}),\\
\la(t) & =   |\kappa|  \frac{ T-t}{|\log (T-t)|^2} (1+ o(1)) ,
\end{aligned} \right.
\end{align*}
as $t \uparrow T$,  for some $\kappa\in \C$, $ \sigma>0$.

\bigskip
The proof of this result takes strong advantage of the symmetry of revolution of the domain. In fact,
restricting the problem to the class of axisymmetric functions, Problem \equ{har flow0} reduces to a
problem only involving the variables $(r,z)$ and the two-dimensional domain $\DD$. We will closely follow the steps
of the main result in \cite{ddw} and make reference to intermediate technical results there.

\medskip
With a very similar proof we can construct simultaneous blow-up in any finite number of disjoint circles $\Gamma$. It would be a
very interesting issue to consider the case $r_0=0$ case in which the singularity would asymptotically collapse onto a point
in the $z$-axis. Lifting the revolution symmetry assumption potentially obtaining other blow-up sets
is a very interesting and difficult issue.

\section{The axially symmetric problem}

In the setting of Theorem \ref{teo1} it is natural to look for solutions which are axially symmetric. More precisely, we
look of a solution of \equ{har flow0} with boundary condition $u_b = {\bf e}_3$ of the form
$$
u(x,t)  :=  \ttt u (r,z) , \quad x= (r e^{i\vartheta}, z) , \quad (r,z) \inn \DD .
$$
where $ \ttt u : \DD\subset \R^2  \to S^2$.
We directly check that in this situation our problem becomes

\be\left \{
\begin{aligned}
\ttt u_t  &=  \ttt u_{rr} + \frac 1r \ttt u_{r}  +   \ttt u_{zz} +  |\nn \ttt u |^2 \ttt u \inn \DD\times (0,T) \\
\ttt u_r  &=0 \onn    \{r=0\} \cap \DD \times (0,T)\\
\ttt u  & =    {\bf e}_3  \onn (\pp \DD\setminus \{r=0\}) \times (0,T) \\
\ttt u(\cdot, 0) &=  \ttt u_0 ,
\end{aligned}\right.
\label{hmf1}\ee
where $\nn \ttt u = ( \ttt u_r, \ttt u_z) $. We want to find a solution $\ttt u(x,z) $
that blows up exactly at the point $q = (r_0,z_0)$ as $t\to T$ in the form
$$
\ttt u(r,z) \approx W \left (\frac {(r,z) -\xi(t) } {\la(t)} \right ), \quad \la(t)\to 0, \quad \xi(t)\to q_0 .
$$
To make a precise ansatz, we consider the family of two-dimensional 1-corrotational harmonic maps
\[
 U_{\la, \xi , \omega}  (r,z)  :=
 Q_\omega\, W (  y ) ,\quad y= \frac {(r,z)- \xi }{\la}, \quad \xi\in \R^2 \   \omega\in \R,\  \la>0,
\]
 where $ W(y)  $ is the canonical 1-corrotational harmonic map \equ{U00}
 and $Q_\omega$ is the $\omega$-rotation matrix
\[
Q_\omega:=   \left [ \begin{matrix}  \cos\omega  &  -  \sin \omega  & 0  \\  \sin\omega  &  \cos\omega & 0  \\  0 & 0 & 1   \end{matrix} \right ].
\]
All these functions satisfy the elliptic equation
\begin{align}
\label{hm0}
U_{rr}+ U_{zz}  + |\nn U|^2 U =0 \inn \R^2, \quad |U|=1 .
\end{align}
For any sufficiently small number $T>0$
we look for an initial datum $u_0$ such that the solution $\ttt u(r,z,t)$ of problem \equ{hmf1}
looks at main order like
\[
  U_{\la(t), \xi(t) , \omega(t)}  (r,z)  =
Q_{\omega(t)}\,  W(  y), \quad y= \frac {(r,z)- \xi(t) }{\la(t)} ,
\]
for certain functions  $\xi(t)$, $\la(t)$ and $\omega (t)$ of class $C^1([0,T])$ such that
\[
\xi(T) =q, \quad \la(T)=0 , \quad \omega(T)= 0.
\]
We consider a first approximation $U (r,z,t)$ which smoothly interpolates $U_{\la(t), \xi(t) , \omega(t)}(r,z,t)$ with $(r,z)\approx q$  and the constant vector ${\bf e}_3$.
Let $\eta(\zeta)$ be a smooth cut-off function so that
$$\eta(\zeta )= \begin{cases} 1 & \hbox{  for $\zeta <1$,}\\ 0 &\hbox{  for $\zeta >2$}. \end{cases}$$
For a fixed small number $\delta>0$ we let
$$ \eta^\delta (r,z) :=  \eta\left ( \frac {|(r,z)-q |}\delta \right ) $$
and set
\be\label{UUU}
U (r,z,t)\ :=\   \, \eta^\delta(r,z) \, U_{\la(t), \xi(t) , \omega(t)}(r,z,t)  \,+\,  (1-\eta^\delta (r,z) ) \, {\bf e}_3 .
\ee

 We shall find values for these functions
so that for a small remainder $v(x,t)$ we have that
$ \ttt u  = U + v $ solves \equ{hmf1}.  The condition $|U+ v|=1$
tells us that $u$ can be written as
\be\label{v}
u(x,t)= U+  \Pi_{U^\perp}\varphi +  a(\Pi_{U^\perp}\varphi) U,
\ee
where $\vp$ is a small function with values into $\R^3$ and we denote
$$
\Pi_{U^\perp} \vp :=   \vp - (\vp\cdot U) U, \quad  a(\zeta)  :=   \sqrt{1 - |\zeta|^2}-1  .
$$
The term $a(\Pi_{U^\perp}\varphi)$ has  a quadratic size in $\vp$. We choose to decompose the remainder $\vp(r,z,t)$ in \equ{v} as the addition
of an ``outer'' part, better expressed in the original variables $(r,z)$, and an ``inner'' part which is supported near the singularity and it is naturally  expressed as function of the slow variable $y$. More precisely, we let
\be\label{aris}
\vp(r,z,t) \  = \  \vp^{out} (r,z,t)  +   \vp^{in} (y,t),  \quad y=\frac{(r,z)-\xi(t)}{\la(t)}
\ee
where
\begin{align}
\nonumber
 \vp^{in} (y,t) \ = &  \  \eta_{R(t)}\left (y  \right) Q_{\omega(t)} \phi(y,t), \quad \phi(y,t)\cdot W(y) \equiv 0
\end{align}
and  $ \eta_R(y) :=  \eta\left (\frac {|y|} R  \right)$
 The function $\phi(y,t)$ is defined for $|y|< 3R(t)$ where $R(t)\to +\infty$ and $\la(t) R(t)\to 0$ as $t\to T$.
With these definitions we see that $\Pi_{U^\perp} \vp^{in}= \vp^{in}$.

\medskip
We choose to the  decompose the outer part $\vp^{out}(x,t)$  in \equ{aris} as
\begin{align}
\nonumber
\vp^{out}(x,t) =  \Phi^0[\omega,\la, \xi] +  Z^*(x,t)  \, + \,  \psi(x,t),
\end{align}
where $ \Phi^0$ and  +  $Z^*(x,t)$ are explicit functions chosen as follows:
$\Phi^0[\omega,\la, \xi]$ is a function (which will be precisely described in the next section) that at main order eliminates the largest slow-decaying  part of the error of approximation $E(r,z,t)$ in \equ{hmf1},
namely $E = S(U)$, where
$$
S(\ttt u) :=  - \ttt  u_t +  \ttt  u_{rr} + \frac {\ttt u_r}r  + \ttt  u_{zz}  + |\nn \ttt u|^2 \tilde u  .
$$
Writing  $p(t) := \la(t) e^{i\omega(t)}$ and using polar coordinates $$(r,z)= \xi(t)+ s e^{i\theta},$$
  we require
$$
 \Phi^0_t  -   \Phi_{rr}^0 -\Phi_{zz}^0  \approx   \frac 2s \left [\begin{matrix} \dot p(t) e^{i\theta} \\ 0   \end{matrix}\right ] \approx E(r,z,t)  .
 $$
With the aid of Duhamel's formula for the standard heat equation, we find that the following function is a good approximate solution:
\begin{align}
\label{defPhi0}
\Phi^0[\omega,\la,\xi] (s,\theta,t) & :=   \left [ \begin{matrix}  \vp^0(s,t) e^{i\theta }
\\ 0 \end{matrix}   \right ]
\\
\nonumber
\vp^0(s,t)
&
= -\int_{-T}^t  \dot p(\tau) s k(z(s),t-\tau) \, d\tau
\\
\nonumber
z(s) & = \sqrt{ s^2+ \la^2} ,\quad k(z,t) = 2\frac{1-e^{-\frac{z^2}{4 t}}}{z^2} ,
\end{align}
where for technical reasons $p(t)$ is assumed to be defined in $[-T,T]$, that is, also for some negative values of $t$.
On the other hand, we let $Z^*:\Omega\times (0,\infty) \to \R^3$ satisfy
\begin{align}
\label{heatZ*}
\left\{
\begin{aligned}
Z_{t}^* &=  \Delta_x Z^*  \inn \Omega\times(0,\infty), \\
Z^*(\cdot ,t)  &=  0\inn  \pp\Omega \times (0,\infty),\\
Z^*(\cdot ,0)  &=  Z^*_0   \inn  \Omega ,
\end{aligned}
\right.
\end{align}
where $Z_0^*(x)$ is is a small, sufficiently regular, axially symmetric function, more precisely
\begin{align}
\label{notationZ0star}
Z_0^*(x) = \ttt Z_0^*(r,z) =  \left [ \begin{matrix} \ttt z_0^*(r,z) \\  \ttt z_{03}^* (r,z) \end{matrix}   \right ] , \quad \ttt z_0^*(r,z) = \ttt z^*_{01}(r,z)  + i \ttt z^*_{02}(r,z), \quad x= (re^{i\vartheta}, z).
\end{align}
 function essentially satisfying
 $$\ttt Z_0^*(q)= 0, \quad  \div  \ttt z^*_0(q)  + i \curl  \ttt z^*_0(q) \ne 0 . $$
 where we denote
\begin{align}
\label{div-curl-z0star}
 \div  \ttt z^*_0(r,z) =   \pp_r \ttt z^*_{01}(r,z) +  \pp_z \ttt  z^*_{02}(r,z), \quad \curl  \ttt z^*_0(r,z)= \pp_r \ttt z^*_{02}(r,z) -  \pp_z \ttt z^*_{01}(r,z).
\end{align}
Of course we have $Z^*(x,t) = \ttt Z^*(r,z,t)$.
Then for $(r,z,t)\in \DD \times (0,T) $ we make the ansatz
\be \left\{ \begin{aligned}
\ttt u(r,z,t)\  = & \  U(r,z,t) \,  +\, v(r,z,t), \\  v(r,z,t)\ = &\   \Pi_{U^\perp} \big(    \eta^\delta\, \Phi^0[\omega, \la , \xi] + \ttt Z^* + \psi\big ) \, + \, \eta_R Q_\omega \phi  + a U  \end{aligned} \right. \label{canave}\ee
for a blowing-up solution $\ttt u(r,z,t)$ of \equ{hmf1},
where $\phi$ and $\psi$ are lower order corrections.
Our task is to find  functions $\omega(t), \la(t) , \xi(t)$, $\psi(x,t)$ and $\phi(y,t)$ as described above, such that the remainder $v$ remains uniformly small.

\medskip
We will define a system of equations that we call the
{\em inner-outer gluing system}, essentially  of the form
\begin{align}
\nonumber
\left\{
\begin{aligned}
\la^2 \phi_t\  & =\  L_ W  [\phi ]\ +\ H[p,\xi, \psi,\phi],  \quad
\phi\cdot  W   =  0 \inn     \R^2\times (0,T)
\\
\psi_t\  &=\   \psi_{rr} + \frac{\psi_r}r + \psi_{zz} \ +\  G[p,\xi, \psi,\phi]\qquad \qquad \qquad\ \,   \inn \DD \times (0,T)
\end{aligned}
\right.
\end{align}
where
\be\label{L}
L_W[ \phi] = \Delta_y \phi + |\nn_y W|^2 \phi + 2(\nn_y\phi \cdot \nn_y W)W , \quad \phi\cdot W = 0
\ee
is the linearized operator for equation \equ{hm0} around $U= W$,
so that if the pair of functions $(\phi(y,t),\psi(x,t))$ solves it then $\ttt u$ given by \equ{canave}
is a solution of \equ{hmf1}.
 The point is to adjust the parameter functions $\omega, \la,\xi$ such that the inner problem can be solved for $\phi(y,t)$ which decays as $|y|\to \infty$. To fix the idea, let us consider the approximate elliptic equation, where time is regarded just as a parameter,
 $$
  L_ W  [\phi ]\ +\ H[p,\xi, 0,0] = 0 \inn \R^2
  $$
As we will discuss, a space-decaying solution $\phi(y,t)$ to this problem exists if a set of orthogonality conditions of the form
$$
\int_{\R^2} H[p,\xi, 0,0](y,t)\, Z(y)\, dy = 0 \foral Z\in \mathcal Z
$$
 where $\mathcal Z $  is a  4-dimensional space constituted by decaying functions $Z(y)$ with $L_W[Z]=0$. These solvability conditions lead to an essentially explicit system
of equations for the parameter functions which will tell us in particular that for some small $\sigma>0$
\begin{align*}
\begin{aligned} p(t) & =  - (\div \ttt z^*_0(q)  + {i\curl \ttt z_0^* (q)} ) \frac {|\log T| }{\log^2 (T-t)} (1+ O( |\log T|^{-1+\sigma} )) ,   \\
\xi_1(t) & =   \sqrt{ r_0^2 + 2 (T-t) } +  O ( (T-t)^{1+\sigma}) \\
\xi_2(t) & =   z_0+   O ( (T-t)^{1+\sigma}),\end{aligned}
\end{align*}
and we recall that we are consistently asking $ \div \ttt z^*_0(q)  + i{\curl \ttt z_0^* (q)} \ne 0$.

\medskip
In the next sections we will carry out in detail the program for the construction sketched above.

\section{The linearized operator around the bubble}

We can represent
 $ W  (y)$ in polar coordinates,
$$
 W  (y) =  \left (\begin{matrix}    e^{ i\theta} \sin {w(\rho )}   \\  \cos {w(\rho )} \end{matrix}    \right ), \quad w(\rho ) = \pi - 2\arctan (\rho ), \quad y= \rho e^{i\theta}.
$$
We notice that
$$
w_\rho = - \frac 2{1+\rho^2} , \quad \sin w = -\rho w_\rho = \frac {2\rho} {1+\rho^2}, \quad \cos w = \frac {\rho^2 -1}{1+\rho^2} .
$$
For the linearized operator  $L_W$ in \equ{L} we have that
$L_ W  [ Z_{lj}] =0 $ where
\begin{align}
\label{ZZ}
\left\{
\begin{aligned}
Z_{01}(y) & = \rho w_\rho(\rho)\, E_1(y)
&
   Z_{02}(y) &= \rho w_\rho(\rho)\, E_2(y)
\\
  Z_{11}(y) & = w_\rho(\rho)\,  [ \cos\theta \, E_1(y) + \sin\theta\,  E_2(y)  ]
&
 Z_{12}(y)  & =w_\rho(\rho)\,  [ \sin\theta  \, E_1(y)  - \cos \theta\,  E_2(y)  ] \\
 Z_{-1,1}(y) &= \rho^2 w_\rho(\rho)[ \cos \theta E_1(y) -\sin \theta E_2(y) ]
&
Z_{-1,2}(y) &=  \rho^2 w_\rho (\rho)[ \sin \theta E_1(y) + \cos \theta E_2(y) ]  .
\end{aligned}
\right.
\end{align}
and
\[
E_1(y) =    \left (\begin{matrix}    e^{ i\theta} \cos w(\rho )   \\ - \sin {w(\rho )} \end{matrix}    \right ), \quad E_2(y) =
\left (\begin{matrix}    ie^{ i\theta}    \\  0 \end{matrix}    \right ) .
\]
These vectors from an orthonormal basis of the tangent space to $S^2$ at the point $ W  (y)$.

\subsection*{The linearized operator at functions orthogonal to \texorpdfstring{$U$}{U}}
We consider the linearized operator $L_U$ analogous to $L_W$ but taken around our basic approximation $U$, that is,
\[
L_U[ \vp] =  \vp_{rr}  + \vp_{zz}  + |\nn U|^2 \vp + 2(\nn \vp \cdot \nn  U)U .
\]
It will be especially significant to compute  the action of $L_U$ on functions with values pointwise orthogonal to $U$.
In what remains of this section we will derive various formulas that will be very useful later on.

\medskip
For an arbitrary function
$\Phi(r,z)$ with values in $\R^3$ we denote the projection
$$
\Pi_{U^\perp} \Phi :=   \Phi - (\Phi\cdot U) U.
$$
A direct computation shows the validity of the following:
\begin{equation}
\nonumber
L_U[\Pi_{U^\perp}\Phi] = \Pi_{U^\perp} (\Phi_{rr}+ \Phi_{zz})  + \ttt L_U[\Phi ]
\end{equation}
where
\[
\ttt L_U[ \Phi ] := |\nn U|^2 \Pi_{U^\perp} \Phi - 2\nn (\Phi \cdot U ) \nn U,
\]
with $\nn = (\pp_r,\pp_z)$
and
$$
\nn (\Phi \cdot U ) \nn U =  \pp_{r} (\Phi \cdot U )\, \pp_{r} U  + \pp_{z} (\Phi \cdot U )\, \pp_{z} U  .
$$
A very convenient expression for $\ttt L_U[ \Phi ]$ is obtained if we use polar coordinates. Writing in complex notation
$$
\Phi(r,z) = \Phi(s,\theta), \quad  (r,z) = \xi + s e^{i\theta},
$$
we find
\begin{align}
\label{Ltilde}
 \ttt L_U[\Phi] =  - \frac 2{\la} w_\rho(\rho)\,  [ (\Phi_s \cdot U)Q_\omega E_1  - \frac 1{s} (\Phi_\theta \cdot U) Q_\omega E_2 ], \quad \rho = \frac s\lambda.
\end{align}

\bigskip
We mention two consequences of formula \equ{Ltilde}.
Let us assume that $\Phi(x)$ is a  $C^1$ function $\Phi : \DD \to \C \times \R$, which we express in the form
\begin{align}
\label{notation-Phi}
\Phi(r,z)\ =\ \left ( \begin{matrix} \vp_1 (r,z) + i \vp_2(r,z)  \\ \vp_3 (r,z)  \end{matrix} \right ).
\end{align}
We also denote $$\vp = \vp_1 + i \vp_2 , \quad \bar \vp = \vp_1  - i \vp_2 $$ and define the operators
$$
\div \vp   = \pp_{r}\vp_1 + \pp_{z}\vp_2, \quad \curl \vp = \pp_{r}\vp_2 - \pp_{z}\vp_1 . $$
Then the following formula holds:
\be
 \ttt L_U [\Phi ] =  \ttt L_U [\Phi ]_0  +  \ttt L_U [\Phi ]_1 +  \ttt L_U [\Phi ]_2\ ,
 \label{Ltilde2} \ee
 where
\begin{align}
\label{Ltilde-j}
\left\{
\begin{aligned}
  \ttt L_U [\Phi ]_0 & =    \la^{-1} \rho w_\rho^2\, \big [\, \div ( e^{-i\omega} \vp)\, Q_\omega  E_1    + \curl ( e^{-i\omega} \vp)\, Q_\omega E_2
  \, \big ]\,
 \\
\ttt L_U [\Phi ]_1 & =
   -\,
  2\la^{-1} w_\rho  \cos w \, \big [\,(\pp_{r} \vp_3) \cos \theta +    (\pp_{z} \vp_3) \sin  \theta \, \big ]\,Q_{\omega} E_1
\\
& \quad - 2\la^{-1}  w_\rho  \cos w  \, \big [\, (\pp_{r} \vp_3) \sin \theta -    (\pp_{z} \vp_3) \cos  \theta \, \big ]\, Q_{\omega}E_2\ ,
\\
\ttt L_U [\Phi ]_2 &=   \quad
\la^{-1} \rho w_\rho^2 \, \big [\, \div (e^{i\omega}\bar \vp)\, \cos 2\theta  -   \curl ( e^{i\omega}\bar \vp)\, \sin 2\theta  \, \big ]\, Q_{\omega} E_1
\\
&  \quad
+  \la^{-1} \rho w_\rho^2 \, \big [\,  \div ( e^{i\omega}\bar \vp)\, \sin 2\theta +   \curl  ( e^{i\omega}\bar \vp)\,  \cos 2\theta   \, \big ]\,Q_{\omega}E_2 .
\end{aligned}
\right.
\end{align}


Another corollary of formula \equ{Ltilde} that we single out is the following:
assume that
$$
\Phi (r,z) =  \left ( \begin{matrix} \phi(s) e^{i\theta} \\ 0  \end{matrix}  \right) , \quad x = \xi + se^{i\theta} , \quad \rho =\frac s\la
$$
where $\phi(s)$ is complex valued.
Then
\begin{align}
\label{uii}
\ttt L_U [\Phi] =    \frac 2\la  w_\rho(\rho)^2   \left [  {\rm Re } \,( e^{-i\omega} {\partial_s \phi(s) } )  Q_\omega E_1   +
 \frac 1s {\rm Im}  \,( e^{-i\omega} \phi(s))  Q_\omega E_2  \right ].
\end{align}
For the proof of the formulas above see \cite{ddw}, section~2.

\hide{
A final result in this section is a computation (in polar coordinates) of the operator $L_U$ acting on a function of the form
\[
\Phi (r,z) =  \vp_1(\rho, \theta) Q_\omega E_1 + \vp_2(\rho, \theta)Q_\omega E_2 ,  \quad x= \xi + \la \rho e^{i\theta}.
\]
We have:
\begin{align}
L_U[\Phi ] & =   \la^{-2}  \left ( \pp^2_\rho \vp_{1} +  \frac  { \pp_\rho\vp_{1} } {\rho }  +  \frac { \pp^2_\theta \vp_{1} } {\rho^2  }+    (2w_\rho ^2 - \frac 1{\rho^2} )\vp_1 - \frac{2}{\rho^2}  \pp_\theta \vp_{2}\cos w   \right ) \,  Q_\omega E_1
\nonumber\\ & \quad
+   \la^{-2}  \left ( \pp^2_\rho \vp_{2} +  \frac  { \pp_\rho\vp_{2} } {\rho }  +  \frac { \pp^2_\theta \vp_{2} } {\rho^2  }+    (2w_\rho  ^2 - \frac 1{\rho^2} )\vp_2  + \frac{2}{\rho^2}  \pp_\theta \vp_{1}\cos w \right ) \, Q_\omega E_2 .
\nonumber
\end{align}
}

\section{The ansatz and the inner-outer gluing system}

The equation we want to solve is $S(\tilde u)=0$, with $\tilde u=U+v$.
A useful observation that we make is that as long as the constraint $|\ttt  u|=1$ is kept at all times and $\ttt u= U+ v$ with $|v|\le \frac 12 $ uniformly, then  for $\ttt u$ to solve  equation \equ{hmf1} it suffices that
\be \label{bU} S(U+v) = b(r,z,t) U \ee  for some scalar function $b$. Indeed, we observe that since $|\ttt u|\equiv 1 $ we have
\[
b\, (U\cdot \ttt  u) = S(\ttt u) \cdot \ttt  u = -  \frac 12 \frac d{dt} {|\ttt u|^2} +  \frac 12 (\pp^2_r + \pp^2_z) {|\ttt u|^2} +
   \frac 1{2r} \pp_r |\ttt u |^2  = 0 ,
\]
and since  $U\cdot u \ge  \frac 12 $, we find that  $b\equiv 0$.

\medskip

We find the following expansion for $S(U+v)$ with $v=\Pi_{U^\perp}\varphi + a(\Pi_{U ^\perp}\varphi)U$:
\[
S( U + \Pi_{U^\perp} \varphi + aU )
=
S(U) - \pp_t \Pi_{U^\perp} \varphi+  L_U(\Pi_{U^\perp} \varphi)  + \frac 1r \pp_r (\Pi_{U^\perp} \varphi )  +  N_U( \Pi_{U^\perp} \vp ) + c(\Pi_{U^\perp} \vp) U   \nonumber
\]
where for $\zeta =\Pi_{U^\perp} \vp $, $a = a(\zeta )$,
\begin{align}
\nonumber
L_U(\zeta )
&=  \zeta_{rr}+ \zeta_{zz}  + |\nabla U|^2 \zeta + 2(\nabla U\cdot  \zeta ) U
\\
\nonumber
N_U( \zeta )
&=
\big[
2 \nn (aU)\cdot \nn (U+ \zeta  )  + 2 \nabla U \cdot \nabla \zeta   + |\nabla \zeta  |^2
+ |\nabla (a U ) |^2 \,
\big] \zeta
- aU_t +   \frac ar \pp_rU
\\
& \quad
\nonumber
+ 2\nn a \nn U  ,
\\
\nonumber
c(\zeta )
& =   a_{rr}+ a_{zz} - a_t  + ( |\nn (U + \zeta  + aU)|^2
- |\nn U|^2 )(1 + a)  -   2\nn U\cdot \nn \zeta +  \frac 1r (\pp_r a) .
\end{align}
Since we just need to have an equation of the form \equ{bU} satisfied,  we find that
$$ \ttt u = U + \Pi_{U^\perp} \varphi + a(\Pi_{U^\perp} \varphi)U $$
solves \equ{hmf1}
if and only if $\vp$ satisfies
\begin{align}
\nonumber
0=
S(U) - \pp_t \Pi_{U^\perp} \varphi+  L_U(\Pi_{U^\perp} \varphi) + \frac 1r \pp_r (\Pi_{U^\perp} \varphi )  +  N_U(\Pi_{U^\perp}\vp ) + b(r,z,t) U ,
\end{align}
for some scalar function $b$.
\hide{
The logic of the construction goes like this:
 As we have said, we  decompose $\varphi$ into the sum of two functions $\vp = \vp^i + \vp^o$, the ``inner'' and ``outer'' solutions and reduce equation \equ{ecuacion}   to solving a  system of two equations in $(\vp^i, \vp^o)$ that we call the inner and outer problems.

\medskip
Using that $V= U_{\la, \xi, \omega}$ satisfies
 $$
V_{rr} + V_{zz}  + |\nn V|^2 V =0
 $$

\medskip
\begin{align*}
Q_{-\omega} L_U[\vp^i]
&=   \la^{-2} \eta   L_ W  [\phi]   +  ((\partial_r^2+\partial_z^2) \eta) \phi + 2 \la^{-1} \nn \eta \nn_y \phi
\\
Q_{-\omega} \vp^i_t
& =  \eta \bigl( \phi_t  - \la^{-1}\dot\la y\cdot \nn_y \phi
- \la^{-1}\dot\xi \cdot\nn_y \phi
+ \dot\omega   Q_{-\omega} \pp_\omega Q_\omega \phi  \bigr) + \eta_t \phi .
\end{align*}
}
We use the ansatz \equ{canave} for $\ttt u$, namely
\begin{align}
\label{upa}
\ttt u(r,z,t)\  =   U +   \Pi_{U^\perp}\vp  \, + \, a(\Pi_{U^\perp}\vp) U, \quad \vp := \Pi_{U^\perp} \big( \eta^\delta\, \Phi^0[\omega, \la , \xi] + \Psi^*\big ) + \eta_R Q_\omega \phi,
\end{align}
where we will later decompose $\Psi^* =  \ttt Z^* + \psi$ for a suitable $\ttt Z^*$.
Equation $S(\ttt u)= 0$  then becomes
\begin{align}
 \label{eqsys1}
0 & =  \la^{-2} \eta Q_\omega  [- \la^2 \phi_t +  \ L_ W  [\phi ] + \la^2 Q_{-\omega} \ttt L_U [\Psi^*] ]
\\
\nonumber
& \quad
+ \eta Q_\omega( \la^{-1}\dot\la  y\cdot \nn_y \phi
+ \la^{-1} \dot\xi \cdot\nn_y \phi  - \dot\omega J \phi )
\\
\nonumber
& \quad
+  \eta^{\delta}    \ttt L_U  [ \Phi^0  ]
+  \eta^{\delta} \Pi_{U ^\perp} [ - \pp_t \Phi^0 + (\partial_r^2+\partial_z^2) \Phi^0  + S(U) ]+ \EE^{out,0}
\\
\nonumber
& \quad - \pp_t \Psi^* +\Delta \Psi^*  +  (1-\eta) \ttt L_U [\Psi^*] +  Q_\omega[((\partial_r^2+\partial_z^2) \eta) \phi + 2  \nn \eta \nn \phi - \eta_t  \phi]
\\
\nonumber
& \quad +
\frac{1}{r} \partial_r
\left( \Pi_{U^\perp} \big( \eta^\delta\, \Phi^0[\omega, \la , \xi] + \Psi^*\big ) + \eta_R Q_\omega \phi
\right)
\\
\nonumber
& \quad +
N_U( \eta Q_\omega \phi  + \Pi_{U^\perp}( \Phi^0 +\Psi^*) ) + ((\Psi^*+ \Phi^0)\cdot U)U_t + b U ,
\end{align}
where
\begin{align*}
\EE^{out,0}
&= \tilde L_U[\eta^\delta \Phi^0]
+ \Pi_{U^\perp}[ (-\partial_t + \partial_r^2 + \partial_z^2)(\delta^\eta \Phi^0) ]
\\
& \quad - \eta^{\delta}    \ttt L_U  [ \Phi^0  ]
-  \eta^{\delta} \Pi_{U ^\perp} [ - \pp_t \Phi^0 + (\partial_r^2+\partial_z^2) \Phi^0   ] + (1-\eta^\delta) S(U).
\end{align*}
We note that from the definition \eqref{UUU} and the fact that $U_{\lambda(t),\xi(t),\omega(t)}$ satisfies the harmonic map equation \eqref{hm0},
we have
\begin{align*}
S(U) &= -U_t + \frac{1}{r} \partial_r U+ \EE^{out,1} , \quad | \EE^{out,1}| + |\nabla  \EE^{out,1}| \leq C \lambda.
\end{align*}
Invoking formulas \eqref{ZZ}  to compute $U_t$ we get
\begin{align*}
 U_t =   \dot\la \pp_\la U_{\la, \xi , \omega}  + \dot\omega  \pp_\omega U_{\la, \xi , \omega} +   \pp_{\xi } U_{\la, \xi , \omega}\cdot \dot \xi
 =   \EE_{0}   + \EE_{1} ,
 \end{align*}
where, setting $  y =  \frac{(r,z)-\xi }{\lambda}= \rho e^{i\theta}  $, we have
\begin{align*}
\EE_{0} (r,z,t)
& =  - Q_\omega [  \frac{\dot \la}{\la} \rho w_\rho(\rho)\,  E_1(y) \, + \,   {\dot \omega } \rho w_\rho(\rho)\,   E_2(y)\, ]
\\
\EE_{1} (r,z,t)
& = -\frac{\dot \xi_{1}}{\la} \, w_\rho(\rho)\, Q_\omega [\ \cos\theta \, E_1(y) + \sin\theta\,  E_2(y)  ]\,
\\
& \quad
- \frac{\dot \xi_{2}}{\la}\,w_\rho(\rho) \,  Q_\omega[ \sin\theta  \, E_1(y)  - \cos \theta\,  E_2(y) \, ].
\end{align*}
The choice \equ{defPhi0} of $\Phi^0$ is so that it cancels $\EE_0$ at main order. The other terms in $S(U)$ behave better,  since $\EE_1$ has faster space decay in $\rho $  and the other terms in $S(U)$ are smaller.
We note that
\begin{align*}
\EE_0(r,z,t) \approx
\ttt \EE_0 (r,z,t) :=  - \frac {2 s } {s^2+\la^2}\left [\begin{matrix} \dot p(t)e^{i\theta} \\ 0   \end{matrix}\right ] ,
\end{align*}
and a  direct computation yields
$$
\Phi^0_t + (\partial_r^2+\partial_z^2) \Phi^0  + \ttt \EE_0    =
\ttt \RR_0 +\ttt \RR_1 ,  \quad \ttt \RR_0 =  \left ( \begin{matrix} \RR_0   \\ 0 \end{matrix}   \right ),\quad
\ttt\RR_1 =  \left ( \begin{matrix} \RR_1   \\ 0 \end{matrix}   \right )
$$
where
\begin{align*}
\RR_0 &:=  - re^{i\theta}   \frac {\la^2}{z^4} \int_{-T}^t  \dot p(\tau)  ( z{k_z} - z^2 k_{zz}) (z(s),t-\tau) \, d\tau \\
\RR_1 & :=
- e^{ i\theta}  {\rm Re}\,( e^{-i\theta} \dot \xi(t))
 \int_{-T}^t  \dot p(\tau) \, k(z(s),t-\tau) \, d\tau
\\
&\qquad
+  \frac r{z^2} e^{i\theta} \, (\la\dot\la(t)  -  {\rm Re}\,( re^{i\theta} \dot\xi(t)) )
\int_{-T}^t  \dot p(\tau) \ {zk_z}(z(s),t-\tau)\,  d\tau.
\end{align*}
We observe that $\RR_1$ is actually a term of smaller order.
Using formulas \equ{Ltilde2},  \equ{uii} and the facts
 $$ \frac {\la^2r }{z^4} = \frac 1{4\la} \rho w_\rho^2, \quad  \frac r {z^2} (1-\cos w) = \frac 1{2\la}  \rho w_\rho^2 ,  $$
we derive an expression for the quantity:
\begin{align*}
&
 \ttt L_U  [ \Phi^0  ]  +
 \Pi_{U ^\perp} [- \pp_t \Phi^0 +(\partial_r^2+\partial_z^2) \Phi^0  + S(U) ]
\\
 & =   \ttt L_U[\Phi^0]  -\EE_1 + \Pi_{U^\perp} [\ttt \EE_0] - \EE_0 +
\Pi_{U^\perp} [\ttt \RR_0] +  \Pi_{U^\perp} [\ttt \RR_1]
\\
&=  \KK_{0}[p,\xi]  + \KK_1[p,\xi] +\Pi_{U^\perp} [\ttt \RR_1]
+ +\Pi_{U^\perp} \Bigl[\frac{1}{r}U + \EE^{out,1} \Bigr]
\end{align*}
where
\begin{align*}
\KK_{0}[p,\xi] =  \KK_{01}[p,\xi] + \KK_{02}[p,\xi]
\end{align*}
with
\begin{align}
\nonumber
\KK_{01}[p,\xi]
&:= - \frac {2}{\la} \rho w_\rho^2
\int_{-T} ^t  \left [ {\rm Re  } \,( \dot p(\tau) e^{-i\omega(t)} )   Q_\omega E_1+
 {\rm Im  } \,( \dot p(\tau) e^{-i\omega(t)} ) Q_\omega E_2   \right ]
\\
\label{K01}
& \qquad \qquad \qquad \qquad \cdot  k(z,t-\tau)  \, d\tau
\\
\nonumber
\KK_{02}[p,\xi]
&
:=  \frac 1{\la} \rho w_\rho^2  \left [  {\dot\la}
-
\int_{-T} ^t  {\rm Re  } \,( \dot p(\tau) e^{-i\omega(t)} ) r k_z(z,t-\tau) z_r \, d\tau\, \right]  Q_\omega E_1
\\
\nonumber
&\quad
-   \frac{1}{4\lambda} \rho w_\rho^2 \cos w  \left [  \int_{-T}^t   {\rm Re}\, ( \dot p(\tau)e^{-i\omega(t) }  )
\, ( z{k_z} - z^2 k_{zz}) (z,t-\tau)\, d\tau\, \right ]  Q_\omega E_1
\\
\label{K02}
&\quad
-    \frac{1}{4\lambda} \rho w_\rho^2  \left [  \int_{-T}^t   {\rm Im }\, ( \dot p(\tau)e^{-i\omega(t) }  )
\, ( z{k_z} - z^2 k_{zz}) (z,t-\tau)\, d\tau\,  \right ]  Q_\omega E_2  ,
\\
\label{K1}
\KK_{1}[p,\xi]
& :=
\frac 1\la  w_\rho \, \big [
\Re \big (  (\dot  \xi_1 - i \dot \xi_2)  e^{i\theta } \big ) Q_\omega E_1
+ \Im \big(  (\dot  \xi_1 - i \dot \xi_2)  e^{i\theta } \big ) Q_\omega E_2       \big ].
\end{align}

We insert this decomposition in equation \equ{eqsys1}
and see that we  will have a solution to the equation if the pair $(\phi,\Psi^*)$ solves
the {\em inner-outer gluing system}
\begin{align}
\label{inner1}
\left\{
\begin{aligned}
\lambda^2 \phi_t  & =  L_ W  [\phi ]
+ \lambda^2  Q_{-\omega} \left[
\tilde L_U  [\Psi^* ]
+ \KK_{0}[p,\xi]+ \KK_{1}[p,\xi] \right]
+ \lambda^2 \chi_{D_{2R}} \frac{1}{r}Q_{-\omega}  \partial_r U
\quad \text{in } D_{2R}
\\
\phi\cdot  W  & =  0   \inn D_{2R}
\\
\phi(\cdot, 0)  & = 0=\phi(\cdot, T),
\end{aligned}
\right.
\end{align}
\smallskip
\begin{align}
\label{outer1}
\partial_t \Psi^* &=  (\partial_r^2+\partial_z^2) \Psi^*  +  g[p,\xi, \Psi^*,\phi]\inn \DD \times (0,T) ,
\end{align}
where $\chi_A$ is characteristic function of a set $A$,
\begin{align}
\label{GG}
 g[p,\xi, \Psi^*,\phi]  & :=
 (1-\eta) \ttt L_U [\Psi^*] + (\Psi^*\cdot U ) U_t
\\
\nonumber
& \quad +
Q_\omega
\bigl( ((\partial_r^2+\partial_z^2) \eta) \phi + 2  \nn \eta \nn \phi - \eta_t  \phi
\bigr)
\\
\nonumber
& \quad +  \eta Q_\omega\bigl( - \dot\omega J \phi  +  \la^{-1}\dot\la  y\cdot \nn_y \phi + \la^{-1} \dot\xi \cdot\nn_y \phi \bigr)
\\
\nonumber
& \quad + (1-\eta)[ \KK_{0}[p,\xi]+ \KK_{1}[p,\xi]] + \Pi_{U^\perp}[ \ttt \RR_1] + ( \Phi^0\cdot U)U_t
\\
\nonumber
& \quad  +
\frac{1}{r} \partial_r
\left( \Pi_{U^\perp} \big( \eta^\delta\, \Phi^0[\omega, \la , \xi] + \Psi^*\big ) + \eta_R Q_\omega \phi
\right)
+ (1-\eta) \frac{1}{r} \partial_r U + \eta^\delta \EE^{out,1} + \EE^{out,0}
\\
\nonumber
& \quad   +  N_U( \eta Q_\omega \phi  + \Pi_{U^\perp}( \Phi^0 +\Psi^*) ) ,
\end{align}
and we denote
$$
D_{\gamma R} = \{(y,t)\in \R^2\times (0,T) \ /\ |y|< \gamma R(t) \}.
$$
Indeed if $(\phi,\Psi^*)$ solves this system, then $\tilde u$ given by \eqref{upa} solves equation \equ{hmf1}.
The boundary condition
$\tilde u=  {\bf e}_3$
on $ (\pp \DD\setminus \{r=0\}) \times (0,T) $
amounts to
$$
\Pi_{U^\perp} [ \Phi^0+  \Psi^*  ] +  a(\Pi_{U^\perp} [U+ \Phi^0+  \Psi^*  ]) U  =  ({\bf e}_3 - U)
$$
and then it suffices that we take the boundary condition for \equ{outer1}:
\begin{align}
\label{bcpsi}
  \Psi^*=  {\bf e}_3 - U -\Phi^0
  \quad \text{on } (\pp \DD\setminus \{r=0\}) \times (0,T) .
\end{align}
We also impose
\begin{align}
\label{bcpsi2}
\partial _r \Psi^* =0 \quad \text{on }  \{r=0\} \cap \DD \times (0,T).
\end{align}

Since we want  $\tilde u(r,z,t)$ to be a small perturbation of $U(x,t)$ when we stand close to $(r_0,z_0,T)$,  it is natural to require that $\Psi^*$ satisfies the final condition
\[
\Psi^* (r_0,z_0,T) = 0.
\]
 This constraint amounts to three Lagrange multipliers when we solve the problem, which we choose to put in the initial condition. Then we assume
\[
\Psi^*\big(r,z,0)  = Z_0^*(x)  +  c_1{\mathbf e_1} +  c_2{\mathbf e_2} + c_3{\mathbf e_3} ,
\]
 where $c_1,c_2,c_3$ are undetermined constants and $Z_0^*(x)$ is a small function for which specific assumptions will later be made.

\section{The reduced equations}
In this section we will informally discuss the procedure to achieve our purpose in particular deriving the order of vanishing of the scaling parameter $\la(t)$ as  $t\to T$.

The main term that couples equations \equ{inner1} and \equ{outer1} inside the second equation is
the linear expression
\[
Q_\omega[((\partial_r^2+\partial_z^2) \eta) \phi + 2  \nn \eta \nn \phi+ \eta_t  \phi],
\]
which is supported in $|y|= O(R)$.
This motivates the fact that we want $\phi$ to exhibit some type of space decay in $|y|$ since in that way $\Psi^*$ will eventually be smaller and in turn that would make the two equations at main order {\em uncoupled}.
Equation \equ{inner1} has the form
\begin{align*}
\la^2 \phi_t  & =  L_ W  [\phi ] +   h[p,\xi, \Psi^*] (y,t)  \inn D_{2R}
\\
\phi\cdot  W  & =  0   \inn D_{2R}
\\
\phi(\cdot, 0) & = 0 \inn B_{2R(0)}  ,
\end{align*}
where, for convenience we assume that $h(y,t)$ is defined for all $y\in \R^2$ extending it outside $D_{2R}$ as
\begin{equation}
\label{HH2}
h[p,\xi, \Psi^*] =
\lambda^2  Q_{-\omega}  \KK_{0}[p,\xi]
+ \lambda^2  Q_{-\omega}
\left[
\tilde L_U  [\Psi^* ]
+ \KK_{1}[p,\xi]
+ \frac{1}{r} \partial_r U
\right] \chi_{D_{2R} } ,
\end{equation}
where
$\KK_0$ is defined in \eqref{K01}, \eqref{K02} and $\KK_1$ in \eqref{K1}.
If $\lambda	(t)$ has a relatively smooth vanishing as $t\to T$ it seems natural that the term $\la^2 \phi_t $ be of smaller order and then the equation is
approximately represented by the elliptic problem
\begin{align}
\label{linearized-elliptic}
L_ W  [\phi ] +   h[p,\xi, \Psi^*]=0, \quad \phi\cdot W  =0  \inn \R^2 .
\end{align}

Let us consider the decaying functions $Z_{lj}(y)$ defined in formula \eqref{ZZ}, which satisfy $L_ W [Z_{lj}]=0$.
If $\phi(y,t)$ is a solution of \equ{linearized-elliptic} with sufficient decay, then necessarily
\be\label{ww1}
\int_{\R^2 }   h[p,\xi, \Psi^*](y,t)\cdot Z_{lj} (y)\, dy = 0  \quad \foral t\in (0,T) ,
\ee
for $l=0,1$, $j=1,2$.
These relations amount to an integro-differential system of equations for $p(t)$, $\xi(t)$, which, as a matter of fact, {\em detemine} the correct values of the parameters so that the solution $(\phi,\Psi^*)$ with appropriate asymptotics exists.

\medskip
We derive next useful expressions for relations \equ{ww1}.   Let us first define
\begin{align}
\label{defB0j}
\mathcal B_{0j} [p] (t) &:=
\frac{\la}{2\pi}
\int_{\R^2}   Q_{-\omega}
[ \KK_{0}[p,\xi]+ \KK_{1}[p,\xi]
] \cdot Z_{0j} (y)\, dy.
\\
\nonumber
\tilde {\mathcal B}_{0j}[p,\xi] & :=
\frac{\la}{2\pi}
\int_{B_{2R}}   Q_{-\omega}
( \frac{1 }{r} \partial_r U )\cdot Z_{0j} (y)\, dy
\end{align}
Using \eqref{K01}, \eqref{K02} the following expressions for $\mathcal B_{01}$, $\mathcal B_{02}$ are readily obtained:
\begin{align*}
\mathcal B_{01} [p](t)
&=
  \int_{-T} ^t  {\rm Re  } \,(\dot p(\tau) e^{-i\omega(t)} )\,
\Gamma_1 \left ( \frac {\la(t)^2}{t-\tau}   \right )  \,\frac{ d\tau}{t-\tau}\,  -2 \dot\la (t)
\\
 \mathcal B_{02}[p](t)
& =
 \int_{-T} ^t  {\rm Im  } \,(\dot p(\tau) e^{-i\omega(t)} )\,
\Gamma_2 \left ( \frac {\la(t)^2}{t-\tau}   \right )  \,\frac{ d \tau}{t-\tau}\,
\end{align*}
where  $\Gamma_j(\tau)$, $j=1,2$  are the smooth functions
defined as follows:
\begin{align*}
\Gamma_1 (\tau)
&
=  - \int_0^{\infty} \rho^3 w^3_\rho \left [  K ( \zeta )
+ 2 \zeta K_\zeta (\zeta ) \frac {\rho^2} { 1+ \rho^2}
-4\cos(w) \zeta^2 K_{\zeta\zeta} (\zeta)
\right ]_{\zeta = \tau(1+\rho^2)}   \, d\rho
\\
\Gamma_2 (\tau) & =
- \int_0^{\infty} \rho^3 w^3_\rho  \left [K(\zeta)   - \zeta^2 K_{\zeta\zeta}(\zeta) \right ]_{\zeta = \tau(1+\rho^2)}
\, d\rho\,
\end{align*}
where
\[
 K(\zeta)  =  2\frac {1- e^{-\frac{\zeta}4}} {\zeta} ,
\]
and we  have used that $\int_0^{\infty} \rho^3 w_\rho^3 d\rho=-2$. Using these expressions we find that
\begin{align}
\nonumber
| \Gamma_l (\tau)- 1|
& \le   C \tau(1+  |\log\tau|) \quad  \hbox{ for }\tau<1 ,
\\
\nonumber
|\Gamma_l (\tau)|
& \le  \frac C\tau\qquad \qquad\hbox{ for }\tau> 1, l=1,2.
\end{align}

Let us define
\begin{align}
\label{defB0-new}
\mathcal B_0[p ] :=
\frac{1}{2}e^{i \omega(t) }
\left(
\mathcal B_{01}[p ]
+  i\mathcal B_{02}[p ] \right)
,\quad
\tilde {\mathcal B}_0[p ] :=
\frac{1}{2}e^{i \omega(t) }
\left(
\tilde {\mathcal B}_{01}[p ]
+  i\tilde {\mathcal B}_{02}[p ] \right)
\end{align}
and
\begin{align}
\nonumber
a_{0j}[p,\xi, \Psi^*]
&:=
- \frac{ \la}{2\pi} \int_{B_{2R}}   Q_{-\omega} \ttt L_U  [\Psi^* ]  \cdot Z_{0j} (y)\, dy, j=1,2,
\\
\label{defA0}
a_{0}[p,\xi, \Psi^*]
& :=
\frac{1}{2}
e^{i \omega(t)} \left( a_{01}[p,\xi, \Psi^*] + i a_{02}[p,\xi, \Psi^*] \right).
\end{align}

\noanot{ 
\begin{ch}
\cb
Maybe change to
\begin{align}
\label{defA0-new}
a_{0}[p,\xi, \Psi^*]
=-\frac{\lambda}{4\pi} e^{i \omega(t)}
\int_{B_{2R} }
\left(
Q_{-\omega(t)} \tilde L_U[\Psi^*] \cdot Z_{01}
+ i Q_{-\omega(t)} \tilde L_U[\Psi^*] \cdot Z_{02}
\right)\,dy .
\end{align}
\end{ch}
} 

Similarly, we let
\begin{align*}
\mathcal B_{1j} [p,\xi ] (t)  & :=
\frac{\la}{2\pi} \int_{\R^2}   Q_{-\omega}
\Bigl[
\KK_{0}[p,\xi]+ \KK_{1}[p,\xi]
+\chi_{D_{2R}} \frac{1}{r}\partial_r U
\Bigr] \cdot Z_{1j} (y)\, dy, j=1,2, \\
\mathcal B_{1} [p,\xi ] (t) & :=   \mathcal B_{11}[p,\xi](t) + i \mathcal B_{12}[p,\xi](t)  .
\end{align*}
At last, we set
\begin{align*}
a_{1j} [p,\xi, \Psi^* ] &:= \frac{\lambda}{2\pi}
\int_{B_{2R}}
Q_{-\omega}
\ttt L_U  [\Psi^* ]
\cdot Z_{1j}(y)  \,dy, \quad j=1,2,
\\
a_1[p,\xi, \Psi^* ] & := - e^{i \omega(t) } ( a_{11}[p,\xi, \Psi^* ] + i a_{12} [p,\xi, \Psi^* ] ) .
\end{align*}

We get that
the four conditions \equ{ww1}  reduce to the system of two complex equations
\begin{align}
\label{eqB0}
\mathcal B_0[p ]& = a_0[p,\xi,\Psi^* ] - \tilde{\mathcal B}_0[p,\xi ]   ,\\
\label{eqB1}
\mathcal B_1[\xi ] & =    a_1[p,\xi,\Psi^* ].
\end{align}

At this point we will make some preliminary considerations on this system that will allow us to find a first guess of the parameters $p(t)$ and $\xi(t)$.
First, we observe that
\begin{align*}
\mathcal B_0[p ]
=   \int_{-T} ^{t-\la^2}    \frac{\dot p(\tau)}{t-\tau}d\tau\, + O\big( \|\dot p\|_\infty \big) ,
\quad
\tilde{\mathcal B}_0[p,\xi] = O(\lambda^{1-\sigma}),
\end{align*}
for any $\sigma>0$.

To get an approximation for $a_0$, let us write
$$
\Psi^* =  \left [ \begin{matrix}\psi^* \\  \psi^*_3  \end{matrix}   \right ] , \quad \psi^* = \psi^*_1 + i \psi^*_2 .
$$
From formula \equ{Ltilde2} we find that
\[
\ttt L_U [\Psi^* ](y)  =  [\ttt L_U]_0 [\Psi^* ]   +  [\ttt L_U]_1 [\Psi^* ]+ [\ttt L_U]_2 [\Psi^* ] ,
\]
 where
\begin{align*}
\la Q_{-\omega} [\ttt L_U]_0 [\Psi^* ]& =   \ \   \rho w_\rho^2\, \big [\, \div ( e^{-i\omega} \psi^*)\,  E_1   + \curl ( e^{-i\omega} \psi^* )\, E_2
  \, \big ]\,
 \\
\la Q_{-\omega}[\ttt L_U]_1 [\Psi^* ] & =
   -\,
  2 w_\rho  \cos w \,  \big [\,(\pp_{r} \psi^*_3) \cos \theta +    (\pp_{z} \psi^*_3) \sin  \theta \, \big ]\, E_1
\\
& \quad - 2 w_\rho  \cos w  \, \big [\, (\pp_{r} \psi^*_3) \sin \theta -    (\pp_{z} \psi^*_3) \cos  \theta \, \big ]\,  E_2\ ,
\\
 \la Q_{-\omega}[\ttt L_U]_2 [\Psi^* ] &=   \quad
 \rho w_\rho^2 \,  \big [\, \div (e^{i\omega}\bar \psi^*)\, \cos 2\theta  -   \curl ( e^{i\omega}\bar \psi^*)\, \sin 2\theta  \, \big ]\,  E_1
\\
&  \quad
+   \rho w_\rho^2 \, \big [\,  \div ( e^{i\omega}\bar \psi^*)\, \sin 2\theta +   \curl  ( e^{i\omega}\bar \psi^*)\,  \cos 2\theta   \, \big ]\,E_2 ,
\end{align*}
and the differential operators in $\Psi^*$  on the right hand sides
are evaluated at $(r,z,t)$ with   $(r,z)= \xi(t)+ \la(t) y$,  $y = \rho e^{i\theta}$ while $E_l= E_l(y)$, $l=1,2$.
From the above decomposition, assuming that $\Psi^*$ is of class $C^1$ in space variable, we find that
\[
a_{0}[p,\xi, \Psi^*] =   [   \div \psi^*+  i\curl \psi^*](\xi,t )  + o(1) ,
\]
where $o(1)\to 0$ as $t\to T$.

Similarly, we have that
\begin{align*}
a_1(p,\xi) & =   2  ( \pp_{r} \psi^*_3 + i\pp_{z} \psi^*_3) (\xi, t) \int_{0}^\infty \cos w \,w_\rho^2 \rho \, d\rho   =  o(1) \ass t\to T,
\end{align*}
since  $\int_0^\infty w_\rho^2 \cos w \rho \, d\rho = 0 $.

Using \equ{K1}, \eqref{ZZ} and the fact that  $\int_0^{\infty} \rho w_\rho^2 d\rho\, =2$ we get
\[
\mathcal B_{1} [\xi ](t)\,  =  \, 2[\, \dot \xi_1(t) + i\dot \xi_2(t)\,]
+ \frac{2}{\xi_1(t)} + O(\lambda^{\sigma}),
\]
for some $\sigma>0$ (actually $\sigma = 2\beta$ where $R \approx \lambda^\beta$.)

 \medskip

Let us discuss informally how to handle  \equ{eqB0}-\equ{eqB1}.
For this we simplify this system in the form
\begin{align}
\nonumber
\int_{-T} ^{t-\la^2}     \frac{\dot p(\tau)}{t-\tau}d\tau
& =
[ \div \psi^*+  i\curl \psi^*](\xi(t),t )  + o(1) + O(\|\dot p\|_\infty)  \\
\dot \xi_1(t) &  =  -\frac{1}{\xi_1(t)} +o(1)\ass t\to T. \label{equB1} \\
\dot \xi_2(t) &  =  o(1)\ass t\to T. \label{equB1-2}
\end{align}
We assume for the moment that the function $\Psi^*(x,t)$ is fixed, sufficiently regular, and we regard $T$ as a parameter that will always be taken smaller if necessary. Recall that we want $\xi(T)=(r_0,z_0) $ where $(r_0,z_0) \in \DD$, $r_0\not=0$ is given, and $\lambda(T)=0$.  Equation  \equ{equB1-2} suggests us
to take  $\xi_2(t) \equiv z_0$ as a first approximation,
while \eqref{equB1} suggest that $\xi_1$ is given at main order by
\[
\xi_1(t) = \sqrt{r_0^2 + 2 (T-t)}.
\]
Neglecting lower order terms, we arrive at the ``clean'' equation for $p(t)= \lambda (t) e^{i\omega(t)}$,
\begin{align}
\label{kuj}
\int_{-T} ^{t-\la(t)^2} \frac{ \dot p(s)}{t-s}ds   =
a_0^*
\end{align}
where $a_0^* = \div \psi^*(q,0 ) + i\curl \psi^*(q,0 )  $.
At this point we make the following assumption:
\begin{align}
\label{div+icurl-not-0}
\div \psi^*(q,0 ) + i\curl \psi^*(q,0 ) \not=0.
\end{align}
We claim that a good approximate solution of \equ{kuj} as $t\to T$  is given by
\[
\dot p(t) =  -\frac {\kappa} {\log^2(T-t)}
\]
for a suitable $\kappa \in \C$. In fact, substituting, we have
\begin{align}
\int_{-T}^{t-\lambda(t)^2} \frac {\dot p(s)}{t-s}\, ds & =
\int_{-T}^{t-  (T-t)  } \frac{ \dot p (s)}{t-s} \, ds +     \, \dot  p (t)\left [ \log (T-t)   - 2\log (\la(t)) \right ]\nonumber
 +   \int_{t-(T-t)   } ^{ t- \la(t)^2}\frac{\dot p(s)-\dot `p(t)}{t-s} ds    \nonumber \\
 \label{formal}
&  \approx
\int_{-T}^{t } \frac{ \dot p (s)}{T-s}\, ds -  \dot p (t) \log (T-t)
\end{align}
as $t\to T$. We see that by the explicit form of $p$,
\begin{align*}
\frac{d}{dt}
\left[
\int_{-T}^{t } \frac{ \dot p (s)}{T-s}\, ds -  \dot p (t) \log (T-t)
\right]=0 ,
\end{align*}
and hence the right hand side of \eqref{formal} is constant.
As a conclusion, equation \equ{kuj}
is approximately satisfied if $\kappa$ is such that
$$
\kappa \int_{-T}^{T} \frac{ \dot p   (s)}{T-s}=a_0^* .
$$
Imposing $p(T)=0$ we  gives us the approximate expression
\[
 p(t) (t) =  a_0^* \frac { |\log T| (T-t)}{\log^2(T-t)}(T-t)\, (1+ o(1)) \ass t\to T.
\]

\section{Solving the inner-outer gluing system}
Our purpose is to determine, for a given $(r_0,z_0)\in \DD$  and a sufficiently small $T>0$, a solution $(\phi,\Psi^*)$ of system \equ{inner1}-\equ{outer1} with a boundary condition of the form
\equ{bcpsi}, \eqref{bcpsi2} such that  $\tilde u(r,z,t)$ given by \equ{upa} blows up with $U(x,t)$ as its main order profile.
This will only be possible for adequate choices
of the  parameter functions $\xi(t)$ and $p(t)= \la(t) e^{i\omega(t)}$.
These functions
will eventually be found by fixed point arguments, but a priori we need to make some assumptions regarding their behavior.

First, we define
\begin{align*}
\lambda_*(t)  = \frac{|\log T| ( T-t)}{ |\log(T-t)|^2}.
\end{align*}
We will assume that for some  positive numbers $a_1,a_2,\sigma$  independent of $T$ the following hold:
\begin{align}
a_1 |\dot \la_* (t)| \le    |\dot p (t)| &  \le    a_2 |\dot \la_* (t)| \foral t\in (0,T),
\nonumber
\\
 |\dot \xi(t) |  & \le   \la_*(t)^{\sigma}\quad\ \foral t\in (0,T).
\nonumber
\end{align}
We also take
\begin{align*}
R(t)  =\la_*(t)^{-\beta},
\end{align*}
where  $\beta \in ( 0,\frac 12)$.

\medskip
To solve the outer equation \eqref{outer1} we will decompose $\Psi^*$ in the form $ \Psi^* =  \tilde Z^* + \psi $ where we let $Z^*:\Omega\times (0,\infty) \to \R^3$ satisfy
\eqref{heatZ*}
with $Z_0^*(x)$  a function satisfying certain conditions to be described below.
Since we would like that  $\tilde u(r,z,t)$ given by \equ{upa} has a blow-up behavior given at main order by that of $U(x,t)$, we will require
\[
\Psi^*(r_0,z_0,T)=0 .
\]
This constraint has three parameters. Therefore we need three ``Lagrange multipliers'' which we include in the initial datum.

\subsection{Assumptions on  \texorpdfstring{$Z_0^*$}{Z0star}}
Let us recall that $Z^*$ solves the  heat equation \eqref{heatZ*} with initial condition $Z_0^*$. We assume first that $Z_0^*$  is axially symmetric  so that
$Z_0^*(x) = \tilde Z_0^*(r,z)$
and use the notation \eqref{notationZ0star} and \eqref{div-curl-z0star}.
A first condition that we require, consistent with \eqref{div+icurl-not-0},  is
$ \div  \tilde z^*_0(q) + i \curl  \tilde z^*_0(q) \not=0$.
In addition we require that $\tilde Z_0^*(r_0,z_0)\approx 0 $ in a non-degenerate way.
We want also $Z^*$ to be sufficiently small, but independently of $T$, so that the heat equation \eqref{heatZ*} is a good approximation of the linearized harmonic map flow far from the singularity.
More precisely, we assume that
for some $\alpha_0>0$ small and some $\alpha_1>0$, all independent of $T$, we have
\begin{align}
\label{condZ0}
\left\{
\begin{aligned}
& \|Z_0^*\|_{C^3(\overline \Omega)}  \le  \alpha_0, \\
& | \tilde Z_0^{*}(r_0,z_0)|   \le  5T,  \\
& |(D \tilde z_0^{*}(r_0,z_0))^{-1}|   \le \alpha_1 ,  \\
& \alpha_0  \leq | \div \tilde z_0^{*}(r_0,z_0)  + i \curl \tilde z_0^{*}(r_0,z_0)  | .
\end{aligned}
\right.
\end{align}
(The notation here is analogous to \eqref{notationZ0star} and \eqref{div-curl-z0star}.)

\medskip
\noanot{ 
\crr

OLD ASSUMPTIONS

More precisely, we consider positive numbers $\alpha_0$, $\alpha_1$, $\alpha_2$, all of them  independent of $T$, with $\alpha_0$ sufficiently small
\begin{align}
\label{conditionsZ0}
\left\{
\begin{aligned}
\|Z_0^*\|_{C^1(\Omega)}  +  T|\log T|^{-1} \|D^2 Z_0^*\|_{L^\infty (\Omega)}
& \le  \alpha_0,
\\
|Z_0^*(q)|
& \le  5T,
\\
|(Dz_0(q))^{-1}|
& \le -\alpha_1,
\\
\div z_0^*(q)  & \le   - \alpha_2.
\end{aligned}
\right.
\end{align}
}  

\medskip

\subsection{Linear theory for the inner problem}
The inner problem \equ{inner1} is written as
\begin{align*}
\left\{
\begin{aligned}
\la^2 \pp_t \phi  & =  L_ W  [\phi ] +   h[p,\xi, \Psi^*]   \inn D_{2R}
\\
\phi  \cdot  W  & =  0   \inn D_{2R}
\\
\phi (\cdot, 0) & = 0 \inn B_{2R(0)}
\end{aligned}
\right.
\end{align*}
where  $h[p,\xi, \Psi^*] $ is given by \equ{HH2}.
To find a good solution to this problem we would like that $ h[p,\xi, \Psi^*] $ satisfies the orthogonality conditions \eqref{ww1}.

We split  the right hand side $h[p,\xi, \Psi^*] $  and the inner solution  into components with different roles regarding these orthogonality conditions.

Recall that
\begin{equation*}
h[p,\xi, \Psi^*] = \lambda^2  Q_{-\omega} \tilde L_U  [\Psi^* ] \chi_{D_{2R} }
+ \lambda^2  Q_{-\omega} \KK_{0}[p,\xi]
+ \lambda^2  Q_{-\omega}  \KK_{1}[p,\xi]  \chi_{D_{2R} }    ,
\end{equation*}
the decomposition of $\tilde L_U$ given in \eqref{Ltilde2}:
\begin{align}
\nonumber
\ttt L_U [\Psi^* ]
=  \tilde L_U [\Psi^* ]_0  +  \tilde L_U [ \Psi^* ]_1 +  \tilde L_U [ \Psi^* ]_2\ ,
\end{align}
with $\tilde L_U[\Phi]_j$ defined in \eqref{Ltilde-j}.
Using the notation \eqref{notation-Phi}, we then define
\begin{align}
\nonumber
\tilde L_U [\Phi ]_1^{(0)} & =
   -\,
  2\la^{-1} w_\rho  \cos w \, \big [\,(\partial_{x_1} \varphi_3(\xi(t),t)) \cos \theta +    (\partial_{x_2} \varphi_3(\xi(t),t))) \sin  \theta \, \big ]\,Q_{\omega} E_1
\\
\nonumber
& \quad - 2\la^{-1}  w_\rho  \cos w  \, \big [\, (\partial_{x_1} \varphi_3(\xi(t),t))) \sin \theta -    (\partial_{x_2} \varphi_3(\xi(t),t))) \cos  \theta \, \big ]\, Q_{\omega}E_2\ .
\end{align}
We then decompose the function $h$ defined in \eqref{HH2}
\[
h = h_1+ h_2 + h_3
\]
where
\begin{align}
\label{def-h1}
h_1[p,\xi, \Psi^*]
&=
\lambda^2  Q_{-\omega} (
\tilde L_U  [\Psi^* ]_0
+\tilde L_U  [\Psi^* ]_2 ) \chi_{D_{2R} }
+ \lambda^2  Q_{-\omega} \KK_{0}[p,\xi] ,
\\
\nonumber
h_2[p,\xi, \Psi^*] &=
\lambda^2  Q_{-\omega} \tilde L_U  [\Psi^* ]_1^{(0)} \chi_{D_{2R} }
+ \lambda^2  Q_{-\omega}  \KK_{1}[p,\xi]  \chi_{D_{2R} },
\\
\nonumber
h_3[p,\xi, \Psi^*] &=   \lambda^2  Q_{-\omega}
( \tilde L_U [\Psi^* ]_1 - \tilde L_U  [\Psi^* ]_1 ^{(0)})  \chi_{D_{2R} }     .
\end{align}

Next we decompose  $\phi = \phi_1+ \phi_2 + \phi_3+\phi_4$. The function $\phi_1$ will solve the inner problem with right hand side $  h_1[p,\xi, \Psi^*]   $ projected so that it satisfies essentially \eqref{ww1}.
The advantage of doing this is that $h_1$ has faster spatial decay, which gives better bounds for the solution.
For this we let, for any function $h(y,t)$ defined in $\R^2 \times (0,T)$ with sufficient decay,
\begin{align}
\label{defCij}
c_{lj}[h](t)  :=     \frac 1 { \int_{\R^2} w_\rho^2 |Z_{lj}|^2  } \int_{\R^2} h (y  ,t)\cdot Z_{lj}(y)\, dy .
\end{align}
Note that  $h[p,\xi, \Psi^*] $  is defined in $\R^2\times (0,T)$, and for simplicity we will assume that the right hand sides appearing in the different linear equations are always defined  in $\R^2\times (0,T)$.

We would like that  $\phi_1$ solves
\begin{align*}
\lambda^2 \pp_t \phi_1  &=  L_ W  [\phi_1 ] +   h_1[p,\xi, \Psi^*]
- \sum_{l=-1}^1
\sum_{j=1}^2 c_{lj}[h_1(p,\xi, \Psi^*)] w_\rho^2 Z_{lj}  \inn D_{2R} ,
\end{align*}
but the estimates for $\phi_1$ are better if the projections $c_{0j}[h(p,\xi, \Psi^*)] $ are modified slightly.

Here is the precise result that we will use later.
We define the norms
\begin{align}
\label{norm-h}
\|h\|_{\nu,a}  =
\sup_{\R^2 \times (0,T)} \  \frac{ |h(y,t)| }{ \lambda_*^\nu (1+|y|)^{-a}} ,
\end{align}
and
\begin{align}
\label{norm-phi1}
\| \phi \|_{*,\nu,a,\delta}
=
\sup_{D_{2R}}
\frac{| \phi(y,t) | + (1+|y|) |\nabla_y \phi(y,t)|}{ \lambda_*^\nu \max(\frac{R^{\delta(5-a)}}{(1+|y|)^3} , \frac{1}{(1+|y|)^{a-2} })} .
\end{align}

\begin{prop}
\label{prop1.0}
Let $a \in (2,3)$, $\delta \in (0,1)$, $\nu>0$.
Assume $\| h \|_{\nu,a}<\infty$. Then there is a solution $\phi =  \TT_{\lambda,1} [h]$,   $\tilde c_{0j}[h]$  of
\begin{align}
\nonumber
\left\{
\begin{aligned}
\lambda^2 \partial_t \phi
& =  L_ W  [\phi ] +   h
-  \sum_{  j=1,2} \tilde c_{0j}[h] Z_{0j} \chi_{B_1}
-  \sum_{ \substack{l=-1,1\\ j=1,2}} c_{lj}[h] Z_{lj} \chi_{B_1}
\quad \text{in } D_{2R}
\\
\phi\cdot  W  & =  0   \quad \text{in } D_{2R}
\\
\phi(\cdot, 0) & = 0 \quad \text{in } B_{2R(0)}
\end{aligned}
\right.
\end{align}
where $c_{lj}$ is defined in \eqref{defCij}, which is  linear in $h$, such that
\begin{align*}
\| \phi \|_{*,\nu,a,\delta}
\leq C \|h\|_{\nu,a}
\end{align*}
and such that
\begin{align}
\nonumber
|c_{0j}[h]  - \tilde c_{0j}[h]| \leq
C  \lambda_*^{\nu}  R^{-\frac{1}{2}\delta(a-2)} \| h \|_{\nu,a}.
\end{align}
\end{prop}

\medskip

The function $\phi_2$ solves the equation with right hand side  $h_2[p,\xi,\Psi^*]$, which is in {\em mode 1}, a notion that we define next.
Let $h(y,t)\in \R^3$, be defined in $\R^2 \times (0,T)$ or $D_{2R}$ with $h\cdot W = 0$. We say that $h$ is a mode $k\in \Z$ if $h$ has the form
\[
h(y,t)= \Re ( \tilde h_k(|y|,t) e^{ik\theta}) E_1 + \Re ( \tilde h_k(|y|,t) e^{ik\theta}) E_2 ,
\]
for some complex valued function $\tilde h_k(\rho,t)$.
Consider
\begin{align}
\label{1.11-mode1}
\left\{
\begin{aligned}
\lambda^2 \partial_t \phi
& =  L_ W  [\phi ] +   h
-  \sum_{  j=1,2} c_{1j}[h] w_\rho^2 Z_{1j}
\quad \text{in } D_{2R}
\\
\phi\cdot  W  & =  0   \quad \text{in } D_{2R}
\\
\phi(\cdot, 0) & = 0 \quad \text{in } B_{2R(0)}
\end{aligned}
\right.
\end{align}

\begin{prop}
Let $a \in (2,3)$, $\delta \in (0,1)$, $\nu>0$.
Assume that $h$ is in mode 1 and $\| h \|_{\nu,a}<\infty$. Then there is a solution $\phi  =\TT_{\lambda,2} [h]$  of \eqref{1.11-mode1}, which is  linear in $h$, such that
\begin{align*}
\| \phi \|_{\nu,a-2}
\leq C \|h\|_{\nu,a} .
\end{align*}
\end{prop}
In the above statemen the norm $\|\phi\|_{\nu,a-2}$ analogous to the one in \eqref{norm-h}, but the supremum is taken in $D_{2R}$.

Another piece of the inner solution, $\phi_3$, will handle  $h_3[p,\xi,\Psi^*]$, which does not satisfy  orthogonality conditions in mode 0.
We will still project it to satisfy the orthogonality condition in mode 1.
Let us consider then \eqref{1.11-mode1} without any orthogonality conditions on $h$ in mode 0.
We define
\begin{align}
\label{norm-starstar}
\|\phi\|_{**,\nu}
= \sup_{D_{2R}} \
\frac{ |\phi(y,t)| + (1+|y|)\left |\nn_y \phi(y,t)\right | }
{ \la_*(t)^{\nu}  R(t)^{2} ( 1+|y| )^{-1}  } .
\end{align}

\begin{prop}
\label{prop02}
Let  $1<a<3$ and $\nu>0$. There exists a $C>0$ such that if  $\|h\|_{a,\nu} <+\infty$ there is a solution $ \phi = \TT_{\lambda,3} [h]$ of \eqref{1.11-mode1}, which is linear in $h$ and satisfies the estimate
$$
\|\phi\|_{**,\nu} \ \le\ C \|h\|_{a,\nu} .
$$
\end{prop}

Note that we allow $a$ to be less than 2 in the previous proposition.

Next we have a variant of Proposition~\ref{prop02} when  $h$ is in mode -1.

\begin{prop}
\label{prop03}
Let  $2<a<3$ and $\nu>0$. There exists a $C>0$ such that for any $h$ in mode -1   with  $\|h\|_{a,\nu} <+\infty$, there is a solution $\phi = \TT_{\lambda,4} [h]$ of problem \eqref{1.11-mode1}, which is linear in $h$ and satisfies the estimate
$$
\|\phi\|_{***,\nu}  \leq C \|h\|_{a,\nu}  ,
$$
where
\begin{align}
\nonumber
\|\phi\|_{***,\nu}
= \sup_{D_{2R}} \
\frac{ |\phi(y,t)| + (1+|y|)\left |\nn_y \phi(y,t)\right | }
{ \la_*(t)^{\nu}  \log(R(t))  } .
\end{align}
\end{prop}

Propositions~\ref{prop1.0}--\ref{prop03} are proved in \cite{ddw}, section~6.

\subsection{The equations for \texorpdfstring{$p = \lambda e^{i\omega}$}{}}
We need to choose the free parameters $p$, $\xi$ so that  $c_{lj}[h(p,\xi, \Psi^*)]=0$ for $l=-1,0,1$, $j=1,2$. This will be easy to do for $l=1$ (mode 1), but mode $l=0$ is more complicated.


To handle $c_{0j}$ we note that by definitions \eqref{HH2}, \eqref{defB0j}, \eqref{defA0}
\begin{align*}
c_{0,j}[h(p,\xi,\Psi^*)]  = \frac{2\pi \lambda }{\int_{\R^2} w_\rho^2 |Z_{0j}|^2}\left(
\mathcal B_{0j}[p] - a_{0j}[p,\xi,\Psi^*]
\right)
\end{align*}
where $B_0$, $a_0$ are defined in \eqref{defB0-new}, \eqref{defA0} and we recall that $p = \lambda e^{i \omega}$.
\noanot{ 
\cb
By the definition of $h[ p,\xi,\Psi^*]$ \eqref{HH2}
\[
h[p,\xi, \Psi^*] = \la^2  Q_{-\omega}\ttt L_U  [\Psi^* ] \chi_{D_{2R} }   + \la^2  Q_{-\omega} [ \KK_{0}[p,\xi]+ \KK_{1}[p,\xi]] ,
\]
and then
\begin{align*}
c_{0j} [  h[p,\xi, \Psi^*]   ]
&=  \frac 1 { \int_{\R^2} w_\rho^2 |Z_{lj}|^2  } \int_{\R^2} h[p,\xi, \Psi^*]   \cdot Z_{0j}(y)\, dy
\\
&=
\frac{\lambda^2}{ \int_{\R^2} w_\rho^2 |Z_{lj}|^2  }
\int_{B_{2R}}    Q_{-\omega}\ttt L_U  [\Psi^* ] \cdot Z_{0j}(y)\, dy
\\
& \quad +
\frac{\lambda^2}{ \int_{\R^2} w_\rho^2 |Z_{lj}|^2  }
\int_{\R^2}Q_{-\omega} [ \KK_{0}[p,\xi]+ \KK_{1}[p,\xi]] \cdot Z_{0j}(y)\, dy
\\
&=\frac{2\pi \lambda }{\int_{\R^2} w_\rho^2 |Z_{0j}|^2}\left(
\mathcal B_{0j}[p] - a_{0j}[p,\xi,\Psi^*] \right).
\end{align*}
} 

So  to achieve $c_{0j}[h(p,\xi, \Psi^*)]=0$ we should solve
\begin{align}
\label{eqAbc}
\mathcal B_0[p ](t)  = a_0[p,\xi,\Psi^* ](t), \quad t\in [0,T],
\end{align}
adjusting   the parameters $\lambda(t)$ and $\omega(t)$.
%
We define the following norms.
Let $I$ denote either the interval $[0,T]$ or $[-T,T]$.
For  $\Theta\in (0,1)$, $l\in \R$ and a continuous function $g:I\to \C$ we let
\begin{align}
\nonumber
\|g\|_{\Theta,l} = \sup_{t\in I} \, (T-t)^{-\Theta} |\log(T-t)|^{l} |g(t)| ,
\end{align}
and for  $\gamma \in (0,1)$, $m \in (0,\infty) $, and $l \in \R$ we let
\begin{align}
\nonumber
[ g]_{\gamma,m,l} = \sup \, (T-t)^{-m}  |\log(T-t)|^{l} \frac{|g(t)-g(s)|}{(t-s)^\gamma} ,
\end{align}
where the supremum is taken over $s \leq t$ in $ I$  such that $t-s \leq \frac{1}{10}(T-t)$.

We have then the following result, whose proof is in
\cite{ddw}, section 13.

\begin{prop}
\label{propIntegralOp}
Let $\alpha ,  \gamma \in (0,\frac{1}{2})$, $l\in \R$, $C_1>1$.
There is $\alpha_0>0$ such that if $\Theta \in (0,\alpha_0)$ and
$m \leq \Theta - \gamma$,
then for  $a:[0,T]\to \C$ is such that
\begin{align}
\label{hypA00}
\left\{
\begin{aligned}
& \frac{1}{C_1} \leq | a(T) | \leq C_1 ,
\\
& T^\Theta |\log T|^{1+\sigma-l} \| a(\cdot) - a(T) \|_{\Theta,l-1}
+ [a]_{\gamma,m,l-1}
\leq C_1 ,
\end{aligned}
\right.
\end{align}
for some $\sigma>0$,
then, for $T>0$ small enough there are two operators $\mathcal P $ and $\Rem$ so that $p = \mathcal P[a]: [-T,T]\to \C$ satisfies
\begin{align}
\label{eq-modified0}
\mathcal B_0[p](t)
= a(t) + \Rem[a](t) , \quad t \in [0,T],
\end{align}
with
\begin{align}
\nonumber
& |\Rem[a](t) |
\\
\nonumber
& \leq  C
\Bigl( T^{\sigma}
+ T^\Theta  \frac{\log |\log T|}{|\log T|}  \| a(\cdot) - a(T) \|_{\Theta,l-1}
+ [a]_{\gamma,m,l-1} \Bigr)
\frac{(T-t)^{m+(1+\alpha ) \gamma}}{  |\log(T-t)|^{l}} ,
\end{align}
for some $\sigma>0$.
\end{prop}

%
%


The idea of the proof og Proposition~\ref{propIntegralOp} is to notice that
\begin{align*}
\mathcal B_0[p ] \approx \int_{-T} ^{t-\la_*(t)^2}     \frac{\dot p(s)}{t-s}ds.
\end{align*}
and decompose
\begin{align}
\nonumber
S_\alpha  [ g]
& :=
g(t) [ - 2\log \lambda_*(t) + (1+\alpha  ) \log (T-t)]
+
\int_{-T} ^{t-(T-t)^{1+\alpha }}  \frac{g(s)}{t-s}ds  ,
\\
\label{defRem}
R_{\alpha }[g]
& :=
-\int_{t-(T-t)^{1+\alpha } } ^{t-\la_*^2}  \frac  {g(t) -g(s)}{t-s} ds.
\end{align}
where $\alpha>0$ is fixed.
We solve a modified equation where in \eqref{eq-modified0} we drop $R_{\alpha }[\dot p]$, and so  the remainder $\Rem$ is essentially $R_{\alpha}[\dot p]$.

Another modification to equations \eqref{eqAbc} that we introduce is to replace $a_0[p,\xi,\Psi^*]$ by its main term.
To do this we write
\begin{align*}
a_0[p,\xi,\Psi]
=a_0^{(0)}[p,\xi,\Psi] + a_0^{(1)}[p,\xi,\Psi] + a_0^{(2)}[p,\xi,\Psi]
\end{align*}
where
\begin{align}
\nonumber
a_0^{(l)}[p,\xi,\Psi]
= -\frac{\lambda}{4\pi}
e^{i\omega}
\int_{B_{2R}}
\left(
Q_{-\omega}\tilde L_U[\Psi]_l \cdot Z_{01} +
i Q_{-\omega}\tilde L_U[\Psi]_l \cdot Z_{02}
\right)\,dy
\end{align}
for $l=0,1,2$.

We define
\begin{align}
\nonumber
c_0^*[p,\xi,\Psi^*](t)
& :=
\frac{ 4 \pi \lambda}{\int_{\R^2} w_\rho^2 |Z_{01}|^2  }
e^{- i\omega}
\Bigl(
\Rem\left[  a_0^{(0)}[p,\xi,\Psi^*]  \right](t)
+a_0^{(1)}[p,\xi,\Psi^*](t)
\\
\nonumber
& \qquad
+ a_0^{(2)}[p,\xi,\Psi^*](t)
\Bigr)  - (c_0[  h[ p,\xi,\Psi^* ]]-\tilde c_0[  h_1[ p,\xi,\Psi^* ] ] )
- \tilde{\mathcal B}_0[p,\xi] ,
\end{align}
and
\begin{align*}
c_{01}^* :=  \Re(c_{0}^*) , \quad
c_{02}^* :=
\Im(c_{0}^*) ,
\end{align*}
where  $\Rem$ is the operator given Proposition~\ref{propIntegralOp} and $\tilde c_0 = \tilde c_{01} + i \tilde c_{02}$ are the operators defined in Proposition~\ref{prop1.0}.

\subsection{The system of equations}

We transform the system \eqref{inner1}-\eqref{outer1} in the problem of finding functions $\psi(r,z,t)$,  $\phi_1(y,t),\ldots,\phi_4(y,t)$,  parameters $p(t) = \la(t)e^{i\omega(t)} $, $\xi(t)$ and constants $c_1,c_2,c_3$
such that the following system is satisfied:
\begin{align}
\label{eq-psi}
\left\{
\begin{aligned}
\psi_t &=  (\partial_r^2+\partial_z^2) \psi  +  g(p,\xi, Z^*+ \psi,\phi_1+ \phi_2+\phi_3+\phi_4)
\inn \DD \times (0,T)
\\
\psi &= ({\bf e}_3 - U) -\Phi^0 \qquad\qquad\ \  \onn (\partial\DD\setminus\{r=0\})\times (0,T)
\\
\partial_r \psi &=0  \qquad\qquad\ \  \onn ( \{r=0\} \cap \DD) \times (0,T)
\\
\psi(\cdot ,0)
&=  (c_1 \,  \mathbf{e_1}  + c_2 \, \mathbf{e_2}  + c_3\,  \mathbf{e_3})\chi +
\ (1-\chi) ({\bf e}_3 - U -\Phi^0)       \inn  \DD
\\
\psi(r_0,z_0,T) &  =  - Z^* (r_0,z_0,T)
\end{aligned}
\right.
\end{align}
\begin{align}
\left\{
\begin{aligned}
\lambda^2  \partial_t \phi_1
&= L_W [\phi_1] + h_1[p,\xi, \Psi^*]
-  \sum_{  j=1,2} \tilde c_{0j}[ h_1[p,\xi, \Psi^*]  ] w_\rho^2 Z_{0j}
\\
\label{eqphi1}
& \qquad
-  \sum_{ \substack{l=-1,1\\ j=1,2}} c_{lj}[ h_1[p,\xi, \Psi^*]  ] w_\rho^2  Z_{lj}
\inn D_{2R}
\\
\phi_1\cdot  W  &=  0   \inn D_{2R}
\\
\phi_1(\cdot, 0) &=0 \inn B_{2R(0)}
\end{aligned}
\right.
\end{align}

\begin{align}
\left\{
\begin{aligned}
\label{eqphi2}
\lambda^2 \partial_t \phi_2
&= L_W [\phi _2]  + h_2[p,\xi, \Psi^*]
 -  \sum_{  j=1,2} c_{1j}[  h_2[p,\xi, \Psi^*]    ] w_\rho^2  Z_{1j}    \inn D_{2R}
\\
\phi_2\cdot  W  & = 0   \inn D_{2R}
\\
\phi_2(\cdot, 0) & =0 \inn B_{2R(0)}
\end{aligned}
\right.
\end{align}

\begin{align}
\label{eqphi3}
\left\{
\begin{aligned}
\lambda^2  \partial_t  \phi_3 &= L_W [\phi_3] +
h_3  -  \sum_{  j=1,2} c_{1j}[  h_3[p,\xi, \Psi^*]    ] w_\rho^2 Z_{1j}
\\
& \quad
+ \sum_{j=1,2} c_{0j}^*[p,\xi,\Psi^*] w_\rho^2 Z_{0j} \inn D_{2R}
\\
\phi_3\cdot  W  & = 0   \inn D_{2R}
\\
\phi_3(\cdot, 0) & =0 \inn B_{2R(0)}
\end{aligned}
\right.
\end{align}

\begin{align}
\label{eqphi4}
\left\{
\begin{aligned}
\lambda^2  \partial_t  \phi_4 &= L_W [\phi_4 ]
+  \sum_{j=1,2}  c_{-1,j}[   h_1[p,\xi, \Psi^*]   ] w_\rho^2 Z_{-1j}
\\
\phi_4\cdot  W  & = 0   \inn D_{2R}
\\
\phi_4(\cdot, t) & = 0 \onn \pp B_{2R(t)}
\\
\phi_4(\cdot, 0) & =0 \inn B_{2R(0)}
\end{aligned}
\right.
\end{align}

\begin{align}
\label{1.3}
c_{0j}[h(p,\xi, \Psi^*)](t) - \tilde c_{0j}[p,\xi,\Psi^*] (t) &= 0 \foral t\in (0,T), \quad j=1,2,
\\
\label{1.4}
c_{1j}[h(p,\xi, \Psi^*)](t)  &=  0 \foral t\in (0,T), \quad j=1,2.
\end{align}
In \eqref{eq-psi}  $\chi$ is a smooth cut-off function with compact support in $\DD$ which is identically 1 on a fixed neighborhood of $(r_0,z_0)$ independent of $T$ and the function $ g(p,\xi, \Psi^*,\phi)$ is given by \equ{GG}.

\medskip
We see that if $(\phi_1,\phi_2,\phi_3,\phi_4,\psi,p,\xi)$ satisfies system \equ{eq-psi}--\equ{1.4}  then the functions
$$
\phi= \phi_1 + \phi_2+\phi_3+\phi_4, \quad \Psi^* = \tilde Z^*+ \psi
$$
solve the outer-inner gluing system  \equ{inner1}--\equ{outer1}.



\subsection{The fixed point formulation}

We consider the inner-outer system including the equations for the parameters $p$ and $\xi$ \eqref{eq-psi}--\eqref{1.4} as a fixed point problem for certain operators that we describe below.
First we define the functional spaces we will use for the functions $\psi,\phi_1,\ldots,\phi_4,p,\xi$.

%
%

For the outer problem \eqref{eq-psi} we define,  given $\Theta>0$, $\gamma \in (0,\frac{1}{2})$  the norm
\begin{align}
\nonumber
\| \psi\|_{\sharp, \Theta,\gamma}
&:=
\lambda_*(0)^{-\Theta}
\frac{1}{|\log T|  \lambda_*(0) R(0) }\|\psi\|_{L^\infty(\Omega\times (0,T))}
+ \lambda_*(0)^{-\Theta} \|\nabla \psi\|_{L^\infty(\Omega\times (0,T))}
\\
\nonumber
&\quad
+
\sup_{\Omega\times (0,T)}   \lambda_*(t)^{-\Theta-1} R(t)^{-1}
\frac{1}{|\log(T-t)|} |\psi(x,t)-\psi(x,T)|
\\
\nonumber
&\quad
+ \sup_{\Omega\times (0,T)} \, \lambda_*(t)^{-\Theta}
|\nabla \psi(x,t)-\nabla \psi(x,T) |
\\
\label{normPsi}
& \quad
+ \sup_{}
\lambda_*(t)^{-\Theta}
(\lambda_*(t) R(t))^{2\gamma}  \frac {|\nn \psi(x,t) -\nn \psi(x',t') |}{ ( |x-x'|^2 + |t-t'|)^{\gamma   }} ,
\end{align}
where the last supremum is taken in the region
\[
x,x'\in \Omega,\quad  t,t'\in (0,T), \quad |x-x'|\le 2 \la_*R(t), \quad  |t-t'| < \frac 14 (T-t) .
\]
Then we define
\begin{align}
\nonumber
F = \{ \psi \in L^\infty(\DD\times (0,T) ) & : \text{$\psi$ is Lipschitz continuous with respecto to $(r,z)$ in $\DD\times (0,T)$}
\\
\label{space-F}
&  \quad  \text{ and } \|\psi\|_{\sharp,\Theta,\gamma}<\infty\}
\end{align}
with the norm $\| \ \|_{\sharp,\Theta,\gamma}$.

For the functions $\phi_i$ in the inner equations \eqref{eqphi1}--\eqref{eqphi4}
we consider the spaces
\begin{align*}
E_1 &= \{ \phi_1 \in L^\infty(D_{2R}) :
\nabla_y \phi_1  \in L^\infty(D_{2R}), \
\|\phi_1\|_{*,\nu_1,a_1,\delta} <\infty \}
\\
E_2 &= \{ \phi_2 \in L^\infty(D_{2R}) :
\nabla_y \phi_2  \in L^\infty(D_{2R}), \
\|\phi_2\|_{\nu_2,a_2}<\infty \}
\\
E_3 &= \{ \phi_3 \in L^\infty(D_{2R}) :
\nabla_y \phi_3  \in L^\infty(D_{2R}), \
\|\phi_3\|_{**,\nu_3} < \infty \}
\\
E_4 &= \{ \phi_4 \in L^\infty(D_{2R}) :
\nabla_y \phi_4  \in L^\infty(D_{2R}), \
\|\phi_4\|_{***,\nu_4} <\infty \}
\end{align*}
and use the notation
\begin{align*}
E &= E_1\times E_2 \times E_3 \times E_4,
\\
\Phi &= ( \phi_1,\phi_2,\phi_3,\phi_4) \in E
\\
\|\Phi\|_E &=
\|\phi_1\|_{*,\nu_1,a_1,\delta}
+\|\phi_2\|_{\nu_2,a_2-2}
+\|\phi_3\|_{**,\nu_3}
+\|\phi_4\|_{***,\nu_4} .
\end{align*}

\medskip
To introduce the space for the parameter $p$,
we recall the  integral operator $\mathcal B_0$  defined in \eqref{defB0-new}, which has the approximate form
\begin{align*}
\mathcal B_0[p ]
=   \int_{-T} ^{t-\la^2}    \frac{\dot p(s)}{t-s}ds\, + O\big( \|\dot p\|_\infty \big).
\end{align*}
Proposition~\ref{propIntegralOp} gives an approximate inverse $\mathcal P$  of the operator $\mathcal B_0$, so that given $a$ satisfying \eqref{hypA00},  $ p := \mathcal P  \left[ a \right] $,
satisfies the equation
\[
\mathcal B_0[ p ]   = a +\Rem[ a] , \quad \text{in }[0,T],
\]
for a small remainder $\Rem[ a]$.
The proof of that proposition  in \cite{ddw} gives a decomposition
\begin{align}
\label{decompP}
\mathcal P[a] = p_{0,\kappa} + \mathcal P_1[a],
\end{align}
where $p_{0,\kappa}$ is defined by
\begin{align}
\nonumber
p_{0,\kappa}(t) =
\kappa |\log T|
\int_t^T \frac{1}{|\log(T-s)|^2}
\,ds
, \quad  t\leq T  ,
\end{align}
$\kappa = \kappa[a] \in \C$,  and the function $ p_1 = \mathcal P_1[a]$ has the estimate
\[
\| p_1 \|_{*,3-\sigma} \leq C |\log T|^{1-\sigma} \log^2(|\log T|) ,
\]
where $\| \ \|_{*,3-\sigma}$ is defined by
\begin{align}
\nonumber
\|g\|_{*,k} = \sup_{t\in [-T,T]}  |\log(T-t)|^{k} |\dot g(t)|,
\end{align}
and $\sigma \in (0,1)$.
This leads us to define the space
\[
 X_1 := \{  p_1 \in C([-T,T;\C]) \cap C^1([-T,T;\C]) \ | \ p_1(T) = 0 , \ \|p_1\|_{*,3-\sigma}<\infty \} ,
\]
with the norm $ \|p_1\|_{*,3-\sigma} $ and represent $p$ by the pair $(\kappa,p_1)$ in the form $p = p_{0,\kappa} + p_1$.

Finally, for the parameter $\xi$ we denote by $\xi^0$ the explicit function
\[
\xi^0(t) = ( \sqrt{r_0^2 + 2(T-t)} , z_0) , \quad t\in [0,T],
\]
and represent $\xi = \xi^0 + \xi^1$ with $\xi^1$ in the space
\begin{align*}
X_2 &= \{  \xi \in C^{1}([0,T];\R^2) \ : \ \dot \xi(T) = 0\}
\end{align*}
with the norm
\[
\| \xi \|_{X_2} = \|\xi\|_{L^\infty(0,T)} + \sup_{t\in (0,T)} \lambda_*(t)^{-\sigma} |\dot \xi(t)|
\]
where $\sigma \in (0,1)$ is fixed.

\medskip
Let $\mathcal B$ denote  the closed subset of $F \times E \times \C \times X_1 \times X_2$ defined by $ (\psi,\Phi,\kappa, p_1,\xi^1) \in \mathcal B$ if:
\begin{align}
\label{def-B}
\left\{
\begin{aligned}
\|\psi\|_F + \|\Phi\|_E & \leq 1 \\
|\kappa-\kappa_0| &  \leq \frac{1}{|\log T|^{1/2}} \\
\| p_1 \|_{*,3-\sigma } &\leq C_0 |\log T|^{1-\sigma} \log^2(|\log T|)
\\
\|\xi^1\|_{X_2} & \leq 1,
\end{aligned}
\right.
\end{align}
where $\kappa_0 = \div \tilde z_0^*(r_0,z_0) + i \curl  \tilde z_0^*(r_0,z_0) $ and $C_0$ is a large fixed constant.

\medskip
Next we define an operator $ \mathcal A : \mathcal B \to F \times E \times \C \times X_1 \times X_2$ so that a fixed point of it will give a solution to the full system \eqref{eq-psi}--\eqref{1.4}.
This operator is defined by
\begin{align*}
\mathcal A = (\mathcal A_0, \mathcal F, \mathcal K, \tilde{\mathcal P}_1,\mathcal X_1),
\end{align*}
where
\begin{align*}
\mathcal A_0 &:\mathcal B \to F ,
\quad
\mathcal F :\mathcal B \to E ,
\\
\mathcal K &:\mathcal B \to \C
\quad
\tilde{\mathcal P}_1 :\mathcal B \to X_1 , \quad
\mathcal X_1 :\mathcal B \to X_2 ,
\end{align*}
and  where $\mathcal A_0$ will hande \eqref{eq-psi}, $\mathcal F$ is related to \eqref{eqphi1}--\eqref{eqphi4} and $\mathcal K $, $ \tilde{\mathcal P}_1 $, $ \mathcal X_1 $ deal with the equations for $p$ and $\xi$, \eqref{1.3}, \eqref{1.4}.

\medskip

To define $\mathcal A_0$, we need first a linear result about the exterior problem \eqref{eq-psi}.
Thus we consider  the inhomogeneous linear heat equation
\begin{align}
\label{heat-eq0}
\left\{
\begin{aligned}
\psi_t  & = (\partial_r^2+\partial_z^2) \psi
+ \frac{1}{r} \partial_r \psi
+ f(r,z,t) \inn \DD \times (0,T) \\
\psi  & = 0 \onn ( \pp \DD \setminus \{ r=0\} ) \times (0,T) \\
\partial_r \psi &= 0  \onn (  \DD \cap \{ r=0\} ) \times (0,T) \\
\psi(r_0,z_0,T)  & = 0  \\
\ \psi(r,z,0)  & =  (c_1 \,  \mathbf{e_1}  + c_2 \, \mathbf{e_2}  + c_3\,  \mathbf{e_3}) \eta_1    \inn  \DD,
\end{aligned}
\right.
\end{align}
for suitable constants $c_1,c_2,c_3$, where $\mathbf{e_1}$, $\mathbf{e_2}$, $\mathbf{e_3}$ are defined in \eqref{e123},
$(r_0,z_0) \in \DD$, $r_0>0$,  and $T>0$ is sufficiently  small.
The fixed smooth cut-off  $\eta_1$ has compact support in $\DD$ and is such that $\eta_1\equiv 1$ in a small neighborhood of $(r_0,z_0)$.
The right hand side is assumed to satisfy $\| f\|_{**}<\infty$ where
\begin{align}
\nonumber
\|f\|_{**} : =   \sup_{ \DD \times (0,T)}  \Big ( 1 + \sum_{i=1}^3 \varrho_i(r,z,t)\, \Big )^{-1}  {|f(r,z,t)|}  ,
\end{align}
and the weights are defined by
\begin{align*}
\varrho_1 & :=   \lambda_*^{\Theta}  (\lambda_* R)^{-1}  \chi_{ \{ s \leq 3R\la_* \} } ,\quad
\varrho_2  := T^{-\sigma_0}  \frac{\lambda_*^{1-\sigma_0}}{s^2}  \chi_{ \{ s \geq  R\la_* \} } , \quad
\varrho_3  := T^{-\sigma_0} ,
\end{align*}
where $s= |(r,z)-(r_0,z_0)|$, $\Theta>0$ and  $\sigma_0>0$ is  small.
(The factor $T^{\sigma_0}$ in front of $\varrho_2$ and $\varrho_3$ is a simple way to have parts of the error small in the outer problem.)
These weights  naturally adapt to the form of the outer error $g$  in \eqref{GG}.
The next lemma gives a solution to  \eqref{heat-eq0} as a linear operator of $f$.
%
\begin{lemma}
\label{lemma3}
Assume  $ \beta \in ( 0,\frac{1}{2}) $, $ \Theta \in (0,\beta)$.
For $T>0$ small there is a linear operator that maps a function $f:\DD \times (0,T) \to \R^3$ with  $\|f\|_{**}<\infty$ into $\psi$, $c_1,c_2,c_3$ so that \eqref{heat-eq0} is satisfied.
Moreover the following estimate holds
\begin{align}
 \label{estPsi0}
\| \psi\|_{\sharp, \Theta ,\gamma}
+ \frac{\lambda_*(0)^{-\Theta}  ( \lambda_*(0) R(0) )^{-1} }{ |\log T| }( |c_1| +|c_2| +|c_3| )  \leq C \|f\|_{**} ,
\end{align}
where $\gamma\in(0,\frac{1}{2})$.
\end{lemma}

\begin{proof} 

We use Lemmas~8.1--8.3 in \cite{ddw}, which put together, can be summarized as follows.

\medskip
{\em
Assume $ \beta \in ( 0,\frac{1}{2}) $, $ \Theta \in (0,\beta)$.
If  $f:\R^2 \times (0,T) \to \R^3$ satisfies  $\|f\|_{**}<\infty$, the solution $\psi$ to
\begin{align}
\label{heat-eq-r2}
\partial_t \psi = ( \partial_r^2 + \partial_z^2)\psi + f(r,z) \quad \text{in }\R^2
\end{align}
given by Duhamel's formula satisfies
\begin{align}
\nonumber
\| \psi\|_{\sharp, \Theta ,\gamma}  \leq C \|f\|_{**} ,
\end{align}
where $\gamma\in(0,\frac{1}{2})$.
}

\medskip
To obtain Lemma~\ref{lemma3} from the above statement we first show that the solution to
\begin{align}
\label{heat-eq0b}
\left\{
\begin{aligned}
\psi_t  & = (\partial_r^2+\partial_z^2) \psi
+ \frac{1}{r} \partial_r \psi
+ f(r,z,t) \inn \DD \times (0,T) \\
\psi  & = 0 \onn ( \pp \DD \setminus \{ r=0\} ) \times (0,T) \\
\partial_r \psi &= 0  \onn (  \DD \cap \{ r=0\} ) \times (0,T) \\
\ \psi(r,z,0)  & =  0   \inn  \DD,
\end{aligned}
\right.
\end{align}
satisfies
\begin{align}
\label{est-psi-ini0}
\| \psi\|_{\sharp, \Theta ,\gamma}  \leq C \|f\|_{**} .
\end{align}
Indeed, let  $\psi[f]$ be the solution of \eqref{heat-eq-r2} given by Duhamel's formula.
We then rewrite the solution $\psi$ of \eqref{heat-eq0b} as $\psi=\eta  \psi[f] + \tilde \psi_1$ where $\eta_1$ is as before.
Then $\tilde \psi_1$  satisfies
\begin{align}
\label{heat-eq0c}
\left\{
\begin{aligned}
\partial_t\tilde\psi_1  & = (\partial_r^2+\partial_z^2) \tilde\psi _1
+ \frac{1}{r} \partial_r \tilde\psi_1
+ \frac{1}{r} \partial_r (\eta_1 \psi[f] )
+ (\partial_r^2+\partial_z^2) \eta_1 \psi[f] +  2\nabla \eta_1\nabla \psi[f]
 \inn \DD \times (0,T) \\
\tilde\psi_1  & = -\psi[f] \onn ( \pp \DD \setminus \{ r=0\} ) \times (0,T) \\
\partial_r \tilde\psi_1 &= 0  \onn (  \DD \cap \{ r=0\} ) \times (0,T) \\
\tilde \psi_1(r,z,0)  & =  0   \inn  \DD,
\end{aligned}
\right.
\end{align}
The function $\tilde \psi_1$ can be regarded as $\tilde \psi_1(r,z,t) = \psi_1(x,t)$ where $\psi_1$ solves a non-homogeneous problem in the three dimensional axially symmetric domain $\Omega$. The estimate $\| \psi[f] \|_{\sharp,\Theta,\gamma} \leq \|f\|_{**}$ gives sufficient control of the terms involving $\psi[f] $ in \eqref{heat-eq0c} so that for $\tilde \psi_1$ we also obtain  $\| \tilde \psi_1 \|_{\sharp,\Theta,\gamma} \leq \|f\|_{**}$.
This proves \eqref{est-psi-ini0}.
Finally, using \eqref{est-psi-ini0} one can show that for the problem \eqref{heat-eq0} there are choices of $c_i$ so that $\psi(r_0,z_0,T)=0$, and these constants satisfy \eqref{estPsi0}.
\end{proof}


Let $\psi = \mathcal U(f)$ be  the operator constructed in Lemma~\ref{lemma3} and  set
\begin{align}
\nonumber
\tilde  g[p,\xi, \Psi^*,\phi]  & := g[p,\xi, \Psi^*,\phi] - \frac{1}{r} \partial_r \psi
\end{align}
with $g$ defined in \eqref{GG}.
We then define
\begin{align}
\label{def-A0}
\mathcal A_0(\psi,\Phi,\kappa, p_1,\xi^1)
= \mathcal U ( \tilde g[p_{0,\kappa}+p_1,\xi^0+\xi^1,Z^* + \psi,\phi] )
\end{align}
where $\phi = \phi_1+\phi_2+\phi_3+\phi_4$ and $\Phi= (\phi_1,\ldots,\phi_4)$.

\medskip

Next we define
\begin{align}
\label{def-F}
\mathcal F (\psi,\Phi,\kappa, p_1,\xi^1) = ( \mathcal F_1(\psi,\Phi,\kappa, p_1,\xi^1) ,  \mathcal F_2(\psi,\Phi,\kappa, p_1,\xi^1) ,
\mathcal F_3(\psi,\Phi,\kappa, p_1,\xi^1)  ,
\mathcal F_4(\psi,\Phi,\kappa, p_1,\xi^1) )
\end{align}
where
\begin{align*}
\mathcal F_1(\psi,\Phi,\kappa, p_1,\xi^1) &=   \mathcal T_{\lambda,1}  (
h_1[p,\xi,  \Psi^* ]  )
\\
\mathcal F_2(\psi,\Phi,\kappa, p_1,\xi^1)&=   \mathcal T_{\lambda,2}  (
h_2[p,\xi,  \Psi^* ] )
\\
\mathcal F_3(\psi,\Phi,\kappa, p_1,\xi^1) &=   \TT_{\lambda,3 }
\Bigl(
h_3[p,\xi,  \Psi^* ]+
\sum_{j=1}^2 c_{0j}^*[p,\xi, \Psi^* ] w_\rho^2 Z_{0j}
\Bigr)
\\
\mathcal F_4(\psi,\Phi,\kappa, p_1,\xi^1)&=   \TT_{\lambda,4 } \Bigl( \sum_{j=1}^2 c_{-1,j}[h_1[p,\xi, \Psi^* ]] w_\rho^2 Z_{-1,j}  \Bigr)	  ,
\end{align*}
where $p = p_{0,\kappa}+p_1$, $\xi= \xi^0+\xi^1$, $\Psi^*  = Z^* + \psi$.

\medskip

To define the operators $\mathcal K$ and $\tilde{\mathcal P}_1$,
we recall that Proposition~\ref{propIntegralOp} gives the decomposition
\eqref{decompP} where $\kappa= \kappa[a]$ and $p_1 = \mathcal P_1[a]$.
We define
\begin{align}
\label{def-K}
\mathcal K(\psi,\Phi,\kappa, p_1,\xi^1)
& = \kappa  \left[ a_0^{(0)}[p,\xi,\Psi^*]\right]
\\
\label{def-tilde-P1}
\tilde{\mathcal P_1} (\psi,\Phi,\kappa, p_1,\xi^1)
&=\mathcal P_1  \left[ a_0^{(0)}[p,\xi,\Psi^*]\right] ,
\end{align}
where, again, $p = p_{0,\kappa}+p_1$, $\xi= \xi^0+\xi^1$, $\Psi^*  = Z^* + \psi$.

\medskip

Finally, we introduce the operator $\mathcal X_1$.
By \eqref{defCij}, \eqref{1.4} is equivalent to
\begin{align*}
\int_{\R^2}
h[p,\xi, \Psi^*]
\cdot Z_{1j}(y)\, dy =0 , \quad t\in (0,T), \ j=1,2,
\end{align*}
and recalling \eqref{HH2}, this is equivalent to
\begin{align}
\nonumber
\dot \xi_j =
- \frac{1}{4\pi}(1+(2R)^{-2})
\int_{B_{2R}} Q_{-\omega} \Bigl(  \tilde L_U[\Psi^*]
+ \frac{1}{r}\partial_r U  \Bigr)
\cdot Z_{1j},
\quad j=1,2.
\end{align}
Then we define
\begin{align}
\label{def-X1}
\mathcal X_1 (\psi,\Phi,\kappa, p_1,\xi^1)
=
(r_0,z_0) +  \int_t^T b(\psi,\Phi,\kappa, p_1,\xi^1) (s) \,ds
\end{align}
with
\begin{align*}
b_{11}(\psi,\Phi,\kappa, p_1,\xi^1)(t)&=
 \frac{1}{4\pi}(1+(2R)^{-2})
 \int_{B_{2R}} Q_{-\omega} \Bigl(  \tilde L_U[\Psi^*]
 + \frac{1}{r}\partial_r U  \Bigr) \cdot Z_{1j}  - \frac{1}{\xi^0_1(t)} \\
b_{12}(\psi,\Phi,\kappa, p_1,\xi^1)(t)&=
 \frac{1}{4\pi}(1+(2R)^{-2})
 \int_{B_{2R}} Q_{-\omega} \Bigl(  \tilde L_U[\Psi^*]
 + \frac{1}{r}\partial_r U  \Bigr) \cdot Z_{1j}  .
\end{align*}

\subsection{Choice of constants}
\label{constants}
We state here the constraints we impose in the parameters involved in the different norms. The values assumed will be sufficient
for the inner-outer gluing scheme to work.

\medskip
\begin{itemize}
\item
$\beta \in (0,\frac 12)$ is so that  $ R(t) = \la_*(t)^{-\beta}$.

\item
$\alpha \in (0,\frac{1}{2})$ appears in Proposition~\ref{propIntegralOp}.
It is the parameter used to define the remainder $\mathcal R_\alpha$ in \eqref{defRem}.

\item
We use the norm $\| \ \|_{*,\nu_1,a_1,\delta}$ \eqref{norm-phi1} to measure the solution $\phi_1$ in \eqref{eqphi1}. Here we will ask that $
\nu_1 \in (0,1)$, $ a_1 \in (2,3)$,  and $\delta > 0$ small and fixed.

\item
We use the norm $\| \ \|_{\nu_2,a_2-2}$ \eqref{norm-h} to measure the solution $\phi_2$ in \eqref{eqphi2}, with   $\nu_2 \in (0,1) $, $a_2 \in (2,3)$.

\item
We use the norm $\| \ \|_{**,\nu_3}$ \eqref{norm-starstar} for the solution $\phi_3$ of \eqref{eqphi3}, with $\nu_3>0$.

\item We use the norm $\| \ \|_{***,\nu_4}$ for the solution $\phi_4$ of \eqref{eqphi4}, with $\nu_4>0$.

\item
We are going to use the norm $\| \ \|_{\sharp,\Theta,\gamma}$ with a parameters $\Theta$, $\gamma$ satisfying some restrictions given below.

\item
We have parameters $m$, $l$ in Proposition~\ref{propIntegralOp}.
We work with $m$ given by
\begin{align*}
 m= \Theta -2\gamma(1-\beta).
\end{align*}
and $l$ satisfying $l<1+2m$.
\end{itemize}

We will assume that
\[
\alpha-1+2\beta>0
\]
which ensures that $m+(1+\alpha)\gamma > \Theta$.

To get the estimates for the outer problem \eqref{eq-psi}, we  need $\beta \in ( 0,\frac{1}{2} )$  , $
\Theta \in
(0,\beta) $
and
\begin{align*}
\Theta <   \min \Bigl( \beta , \frac{1}{2} - \beta ,
\nu_1-1+\beta(a_1-1) ,  \nu_2-1+\beta(a_2-1) , \nu_3-1,\nu_4-1+\beta  \Bigr)
\end{align*}
\begin{align*}
\Theta < \min
\Bigl(  \nu_1 - \delta \beta (5-a_1) -\beta , \nu_2 - \beta , \nu_3-3\beta , \nu_4 - \beta  \Bigr) .
\end{align*}
Also to control the nonlinear terms in \eqref{eq-psi} we need $\delta>0$ in  $\| \ \|_{*,\nu_1,a_1,\delta}$ to be small.
%
To find $\Theta$ in the range above we need
\begin{align*}
\nonumber
\nu_1 & > \max\bigl(1-\beta (a_1-1),\delta  \beta (5-a_1) - \beta \bigr)
, &
\nu_2 & >  \max\bigl(1-\beta (a_2-1), \beta \bigr) ,
\\
\nu_3 & > \max ( 1 , 3\beta) , &
\nu_4 & > \max(1-\beta,\beta) .
\end{align*}

To solve the inner system given by equations
\eqref{eqphi1},
\eqref{eqphi2},
\eqref{eqphi3}, and
\eqref{eqphi4}
we will need
\begin{align}
\nonumber
\nu_1 & < 1 , &
\nonumber
\nu_2 & < 1-\beta(a_2-2) ,  \\
\nonumber
\nu_3 &<  \min\bigl(
1+\Theta+\sigma_1  ,
1+\Theta + 2\gamma \beta ,
\nu_1 + \frac{1}{2} \delta \beta ( a_1-2 ) \bigr) ,
&
\nonumber
\nu_4 & < 1,
\end{align}
where $\sigma_1  \in (0,\gamma(\alpha-1+2\beta))$.

\subsection {The proof of Theorem~\ref{teo1}}
Let us consider the operator
\be
\begin{aligned}
\mathcal A = (\mathcal A_0,\mathcal F,\mathcal K,\tilde{\mathcal P}_1,\mathcal X_1)
\end{aligned}\label{operador1}\ee
where $\mathcal A_0$, $\mathcal F$, $\mathcal K$, $\tilde{\mathcal P}_1$, $\mathcal X_1$ are given in \eqref{def-A0}, \eqref{def-F}, \eqref{def-K}, \eqref{def-tilde-P1}, \eqref{def-X1}.

The proof of Theorem~\ref{teo1} consists in showing that $\mathcal A:\mathcal B \subset F \times E \times \C \times X_1 \times X_2 \to F \times E \times \C \times X_1 \times X_2$ has a fixed point, where $\mathcal B$ is defined by \eqref{def-B}.
We do this using the Schauder fixed point theorem. The estimates needed to show that $\mathcal A$ maps $\mathcal B$ into itself and the compactness are obtained in a similar way.
They are based on the following estimates for the operators
e $\mathcal A_0$, $\mathcal F$, $\mathcal K$, $\tilde{\mathcal P}_1$, $\mathcal X_1$. We claim that if $(\psi,\Phi,\kappa,p_1,\xi^1) \in \mathcal B$ then
\begin{align}
\label{estimates}
\left\{
\begin{aligned}
\|  \mathcal A_0(\psi,\Phi,\kappa,p_1,\xi^1) \|_{\sharp,\Theta,\gamma}
& \leq C T^\sigma
\\
\| \mathcal F (\psi,\Phi,\kappa,p_1,\xi^1) \|_E
& \leq C T^\sigma
\\
|\mathcal K  (\psi,\Phi,\kappa,p_1,\xi^1) - \kappa_0| & \leq \frac{C}{|\log T|}
\\
\| \tilde{\mathcal P}_1(\psi,\Phi,\kappa,p_1,\xi^1) \|_{*,3-\sigma} & \leq C
|\log T|^{1-\sigma} \log^2( |\log T| )
\\
\| \mathcal X_1(\psi,\Phi,\kappa,p_1,\xi^1) \|_{X_2} & \leq C T^\sigma.
\end{aligned}
\right.
\end{align}

We give below the proof of some of the estimates stated above.
We first show that for $(\psi,\Phi,\kappa,p_1,\xi^1)\in \mathcal B$,
\begin{align}
\nonumber
\|  \mathcal A_0(\psi,\Phi,\kappa,p_1,\xi^1) \|_{\sharp,\Theta,\gamma}
\leq C T^\sigma .
\end{align}
For the proof let us write
$ \tilde  g = g_1 + g_2 + g_3 + g_4 + g_5  $ where
\begin{align*}
g_1 & =
Q_\omega
\bigl( ((\partial_r^2+\partial_z^2) \eta) \phi + 2  \nn \eta \nn \phi - \eta_t  \phi
\bigr)
\\
& \quad
+ \eta Q_\omega\bigl( - \dot\omega J \phi  +  \la^{-1}\dot\la  y\cdot \nn_y \phi + \la^{-1} \dot\xi \cdot\nn_y \phi \bigr)
\\
g_2
& =  (1-\eta) \ttt L_U [\Psi^*] + (\Psi^*\cdot U ) U_t
\\
g_3 & =
(1-\eta)[ \KK_{0}[p,\xi]+ \KK_{1}[p,\xi]] + \Pi_{U^\perp}[ \ttt \RR_1] + ( \Phi^0\cdot U)U_t ,
\\
g_4 & = N_U( \eta Q_\omega \phi  + \Pi_{U^\perp}( \Phi^0  + \Psi)^* ) \\
g_5&=\frac{1}{r} \partial_r
\left( \Pi_{U^\perp} \big( \eta^\delta\, \Phi^0[\omega, \la , \xi] + Z^*\big ) + \eta_R Q_\omega \phi
\right)
+ (1-\eta) \frac{1}{r} \partial_r U + \eta^\delta \EE^{out,1} + \EE^{out,0} .
\end{align*}
We claim that
\begin{align}
\nonumber
\|g_1\|_{**} \leq C  T^\sigma  \|\Phi\|_E ,
\end{align}
for some $\sigma>0$.
Indeed, we have
\begin{align*}
| (\partial_r^2+\partial_z^2) \eta \phi_1  |
&\leq C \lambda_*^{\nu_1-2} R^{-a_1} \chi_{[ |x-q| \leq 3 \lambda_* R]}
\|\phi_1\|_{*,\nu_1,a_1,\delta}
\\
| (\partial_r^2+\partial_z^2) \eta \phi_2  |
&\leq C \lambda_*^{\nu_2-2} R^{-a_2} \chi_{[ |x-q| \leq 3 \lambda_* R]}
\|\phi_2\|_{\nu_2,a_2-2}
\\
| (\partial_r^2+\partial_z^2) \eta \phi_3  |
&\leq C \lambda_*^{\nu_3-2} R^{-1} \chi_{[ |x-q| \leq 3 \lambda_* R]}
\|\phi_3\|_{**,\nu_3}
\\
| (\partial_r^2+\partial_z^2) \eta \phi_4  |
&\leq C \lambda_*^{\nu_4-2} R^{-2} \log R \chi_{[ |x-q| \leq 3 \lambda_* R]}
\|\phi_4\|_{***,\nu_4}  .
\end{align*}
If
\begin{align*}
\Theta < \min( \nu_1-1+\beta(a_1-1) ,  \nu_2-1+\beta(a_2-1) , \nu_3-1,\nu_4-1+\beta ) ,
\end{align*}
we find that  for any $j=1,2,3,4$:
\begin{align*}
| \phi_j (\partial_r^2+\partial_z^2) \eta |\leq C T^\sigma \lambda_*^{\Theta-1+\beta}
\chi_{[ |x-q| \leq 3 \lambda_* R]}
\| \Phi \|_E,
\end{align*}
for some $\sigma>0$.
Then we have
\[
\| Q_\omega  ((\partial_r^2+\partial_z^2) \eta) \phi  \|_{**} \leq C
T^\sigma
\|\Phi\|_E
\]
and similarly
\[
\| ( \partial_t \eta)  Q_\omega \phi \|_{**}
+\|   Q_\omega  \la^{-1} \nabla\eta\nabla_y \phi   \|_{**}
\leq  C T^\sigma
\|\Phi\|_E .
\]

The other terms $g_2$, $g_3$, $g_4$, $g_5$ can be estimated in the same way.
In the estimate for $g_2$ it is important to have the property that $\Psi^* = Z^* + \psi$ vanishes at $(r_0,z_0,T)$.

\medskip

Next we estimate  the operator $\mathcal F_1$. The other operators $\mathcal F_2,\ldots,\mathcal F_4$ are handled similarly.
We claim that for  $(\psi,\Phi,\kappa,p_1,\xi^1)\in \mathcal B$,  we have
\begin{align}
\label{estF1-1}
\|  \mathcal F_1(\psi,\Phi,\kappa,p_1,\xi^1)  \|_{*,a,_1,\nu_1}
&  \leq C \lambda_*(0)^{\sigma}
( \| \psi \|_{\sharp,\Theta,\gamma}
+ \| \dot p \|_{L^\infty(-T,T)}
+ \|  Z_0 \|_{C^2} ) .
\end{align}
Indeed,  by Proposition~\ref{prop1.0} we have
\begin{align}
\label{from-prop1.0}
\|  \mathcal F_1(\Phi)  \|_{*,\nu_1,a_1,\delta}
& \leq C \| h_1[p,\xi,  \Psi^* ] \|_{\nu_1,a_1}.
\end{align}
From the definition of $h_1$ \eqref{def-h1} and recalling that
$
\Psi^* = Z^* +  \psi
$
we get
\begin{align*}
&
\| h_1[p,\xi,  \Psi^* ] \|_{\nu_1,a_1}
\\
& \quad \leq
\| \lambda^2  Q_{-\omega}   (\tilde L_U  [ \psi ]_0
+\tilde L_U  [ \psi ]_2  ) \chi_{D_{2R} }  \|_{\nu_1,a_1}
+\| \lambda^2  Q_{-\omega}
 (\tilde L_U  [ Z^* ]_0+\tilde L_U  [ Z^* ]_2  )
 \chi_{D_{2R} }  \|_{\nu_1,a_1}
\\
& \quad \quad
+  \|  \lambda^2  Q_{-\omega}  \KK_{0}[p,\xi] \|_{\nu_1,a_1} .
\end{align*}
We claim that for $j=0$ and $j=2$:
\begin{align}
\label{phi1RHS1-0}
\| \lambda^2  Q_{-\omega} \tilde  L_U  [ \psi ]_j \, \chi_{D_{2R} }
 \|_{\nu_1,a_1}
&\leq
C T^\sigma
\lambda_*(0)^\Theta \| \psi \|_{\sharp,\Theta,\gamma}
\end{align}
Indeed,  from \eqref{Ltilde-j} we get, for $j=0$ and $j=2$:
\begin{align*}
| \lambda^2  Q_{-\omega} \tilde  L_U  [ \psi ]_j|
& \leq C \frac{\lambda_*}{(1+|y|)^3} \|\nabla \psi  \|_{L^\infty} .
\end{align*}
We use  $ \nu_1 <1 $  and $a_1<3$ to estimate for $|y|\leq 2 R$
\begin{align}
\nonumber
\frac{\lambda_*}{(1+|y|)^3}
& \leq
\frac{\lambda_*^{\nu_1}}{(1+|y|)^{a_1}}  \lambda_*(0)^{1-\nu_1}  .
\end{align}
Then for $|y|\leq 2 R$ and $j=0,2$:
\begin{align*}
| \lambda^2  Q_{-\omega} \tilde  L_U  [ \psi ]_j |
& \leq C \frac{\lambda_*^{\nu_1}}{(1+|y|)^{a_1}}
\lambda_*(0)^{1-\nu_1}
\|\nabla \psi \|_{L^\infty}
 \leq C \frac{\lambda_*^{\nu_1}}{(1+|y|)^{a_1}}
\lambda_*(0)^{1-\nu_1}
\lambda_*(0)^\Theta \| \psi \|_{\sharp,\Theta,\gamma},
\end{align*}
and  \eqref{phi1RHS1-0} follwos.
Next we claim that
\begin{align}
\label{phi1RHS2}
\| \lambda^2  Q_{-\omega} \tilde  L_U  [Z^* ]_j \chi_{D_{2R} }  \|_{\nu_1,a_1}
\leq
C T^\sigma \| Z_0 \|_{C^2},
\end{align}
for $j=0,2$ and some $\sigma>0$.
Indeed, we use the assumption \eqref{condZ0} and standard estimates for the heat equation to obtain for $j=0,2$:
\begin{align*}
|  \la^2  Q_{-\omega}\ttt L_U  [ Z^* ]_j \,  \chi_{D_{2R} } |
\leq C \frac{\lambda_*}{(1+\rho)^3}
\| Z_0 \|_{C^2(\overline \Omega)} ,
\end{align*}
Since $\nu_1<1$, we get
\begin{align*}
\|  \la^2  Q_{-\omega}\ttt L_U  [  Z^*    ]_j \, \chi_{D_{2R} }  \|_{\nu_1,a_1}
& \leq C \lambda_*(0)^{1-\nu_1} \| Z_0 \|_{C^2(\overline \Omega)} .
\end{align*}
This implies \eqref{phi1RHS2}.
Next we estimate $ \lambda^2  Q_{-\omega} \KK_{0}[p,\xi] $.
We claim that
\begin{align}
\label{phi1RHS3}
\|  \lambda^2  Q_{-\omega}  \KK_{0}[p,\xi] \|_{\nu_1,a_1}
& \leq C T^\sigma \| \dot p \|_{L^\infty(-T,T)}.
\end{align}
Indeed, consider $\KK_{01}$ given in \eqref{K01}.
We have
\begin{align*}
|\lambda^2 Q_{-\omega} \KK_{01}[p,\xi] |
&
\leq C \frac{\lambda_*}{(1+\rho)^3} \int_{-T} ^t  |  \dot p(s)    k(z,t-s) | \, ds .
\end{align*}
A direct computation shows that
\begin{align*}
\|  \la^2  Q_{-\omega} \tilde  L_U  [  \KK_{01}[p,\xi]   ] \chi_{D_{2R} }  \|_{\nu_1,a_1}
& \leq C \lambda_*(0)^{1-\nu_1} \| \dot p \|_{L^\infty(-T,T)}
\\
& \leq CT^\sigma \| \dot p \|_{L^\infty(-T,T)} ,
\end{align*}
for some $\sigma>0$.
The estimate for $\KK_{02}$ is similar, and we obtain \eqref{phi1RHS3}.
Combining \eqref{phi1RHS1-0},  \eqref{phi1RHS2},  and \eqref{phi1RHS3} we finally obtain
\begin{align}
\nonumber
& \| h_1[p,\xi,  \Psi^* ] \|_{\nu_1,a_1}
\leq C T^\sigma( \| \psi \|_{\sharp,\Theta,\gamma} + \| \dot p \|_{L^\infty(-T,T)}
+  \| Z_0^* \|_{C^2}  ) ,
\end{align}
and combining with \eqref{from-prop1.0} we get \eqref{estF1-1}.  

Compactness of the operator $\mathcal A $ in \equ{operador1} is proved using suitable variants of \eqref{estimates}.
Indeed, the previous computations show that if we vary the parameters $\Theta,\gamma,\nu_j,a_j,\delta, \sigma$  of the norms  slightly, so that the restrictions in \S \ref{constants} are kept, then we still obtain \eqref{estimates} where the norms in the left hand side are defined with the new parameters while $\mathcal B$ is defined with the old parameters.
More precisely, one can show, for example, that if  $\Theta',\gamma'$ are fixed close to $\Theta,\gamma$, then for $(\psi,\Phi,\kappa,p_1,\xi^1)\in \mathcal B$ (this set defined still with $\Theta, \gamma, \ldots$) we get
\begin{align}
\nonumber
\|  \mathcal A_0(\psi,\Phi,\kappa,p_1,\xi^1) \|_{\sharp,\Theta',\gamma'}
\leq C T^\sigma ,
\end{align}
(for a possibly  different $\sigma>0$).
Then one proves that if $\gamma<\gamma'$, $\Theta'-\Theta>2(\gamma'-\gamma)$ one has a compact embedding in the sense that if  $(\psi_n)_n$ is a bounded sequence in the norm $\| \ \|_{\sharp,\Theta',\gamma'} $, then for a subsequence it converges in the norm  $\| \ \|_{\sharp,\Theta,\gamma} $.
This compact embedding is a direct consequence of a standard diagonal argument using Ascoli's theorem, and examining the estimates  for a uniform smallness control of its values near time $T$. 
Similar statements hold for the other components $\Phi,\kappa,p_1,\xi^1$.
The proof is concluded.  \qed

\medskip

\noindent
{\bf Acknowledgements:}

The  research  of J.~Wei is partially supported by NSERC of Canada. J.~D\'avila and M.~del Pino have been supported by grants Fondecyt  1130360,  1150066, Fondo Basal CMM (AFB170001).

\end{document}